\documentclass[11pt,preprint]{imsart}
\usepackage{amsmath,amsthm,verbatim,amssymb,amsfonts,amscd, graphicx}
\usepackage{amssymb,bm}
\usepackage{mathrsfs}
\usepackage{hyperref}
\usepackage{graphicx}
\hypersetup{
    colorlinks,
    citecolor=black,
    filecolor=black,
    linkcolor=blue,
    urlcolor=black
}
\usepackage{url}
\usepackage{verbatim}
\usepackage{constants}
\usepackage{enumerate}
\newcommand{\beas}{\begin{eqnarray*}}
\newcommand{\enas}{\end{eqnarray*}}
\newcommand{\bea}{\begin{eqnarray}}
\newcommand{\ena}{\end{eqnarray}}
\newcommand{\bms}{\begin{multline*}}
\newcommand{\ems}{\end{multline*}}
\newcommand{\bels}{\begin{align*}}
\newcommand{\enls}{\end{align*}}
\newcommand{\bel}{\begin{align}}
\newcommand{\enl}{\end{align}}

\newcommand{\ignore}[1]{}

\newtheorem{theorem}{Theorem}[section]
\newtheorem{corollary}{Corollary}[section]

\newtheorem{lemma}{Lemma}[section]
\newtheorem{definition}{Definition}[section]

\newtheorem{assumption}{Assumption}[section]

\def\blfootnote{\xdef\@thefnmark{}\@footnotetext}

\newcommand{\expect}[1]{\mathbb{E}{\l[#1\r]}}
\newcommand{\mf}[1]{\mathbf{#1}}

\newcommand{\dotp}[2]{\left\langle#1,#2\right\rangle}

\newcommand{\mb}{\mathbb}
\newcommand\argmin{\mathop{\mbox{argmin}}}

\def\r{\right}
\def\l{\left}

\begin{document}
\begin{frontmatter}
\title{\Large Primal-Dual Frank-Wolfe for Constrained Stochastic Programs with Convex and Non-convex Objectives}
\runtitle{Convex and Non-convex}

\begin{aug}
\author{\fnms{Xiaohan} \snm{Wei}\thanksref{t1}\ead[label=e1]{xiaohanw@usc.edu}}
\and
\author{\fnms{Michael} \snm{J. Neely}\thanksref{t1}\ead[label=e3]{mikejneely@gmail.com}}

\thankstext{t1}{Department of Electrical Engineering, University of Southern California. This work is supported in part by NSF grant CCF-1718477.}
\runauthor{Neely and Wei}

\affiliation{University of Southern California}
\printead{e1,e3}
\end{aug}

\maketitle

\begin{abstract}
We study constrained stochastic programs where the decision 
vector at each time slot cannot be chosen freely but is tied to 
 the realization of an underlying random state vector. The goal is to 
minimize a general objective function subject to linear constraints. 
A typical scenario where such programs appear is  opportunistic 
scheduling over a network of time-varying channels, where the random state 
vector is the channel state observed, and the control vector is the transmission 
decision which depends on the current channel state.  
We consider a primal-dual type Frank-Wolfe algorithm that 
has a low complexity update during each slot and that learns to make
efficient decisions without prior knowledge of the 
probability distribution of the random state vector. 
We establish convergence time guarantees for the case of both convex 
and non-convex objective functions.  We also emphasize application of the 
algorithm to non-convex opportunistic 
scheduling and distributed non-convex stochastic optimization over a connected graph. 
\end{abstract}

\end{frontmatter}

\section{Introduction}
Consider a slotted system with time $t\in\{0,1,2,\cdots\}$.
Let $\l\{S[t]\r\}_{t=0}^\infty$ be a sequence of independent and identically distributed (i.i.d.) system state vectors that take values in some set 
$\mathcal S \subseteq \mathbb R^m$, where $m$ is a positive integer. The state vectors $S[t]$ have a probability distribution function
$F_S(s) := Pr(S\leq s)$ for all $s\in\mathbb R^m$ (the vector inequality $S \leq s$ is taken to be entrywise).  Every time slot $t$, the controller observes a realization $S[t]$ and selects a decision vector 
$x_t\in\mathcal X_{S[t]}$, where $\mathcal{X}_{S[t]}$ is a set of decision options available when the system state is $S[t]$. Assume that for all slots $t$, the set   $\mathcal X_{S[t]}$ is a compact subset of some larger compact and convex set 
$\mathcal{B}\subseteq \mathbb{R}^d$.

For each $t\in \{0, 1, 2, \ldots\}$, let $\mathcal{H}_t$ be the system history up to time slot $t$ (not including slot $t$).  Specifically, 
$\mathcal{H}_t$ consists of the past information $\{S[0], \ldots, S[t-1], x_0, \ldots, x_{t-1}\}$. 
We have the following definition:
\begin{definition}
A randomized stationary algorithm is a method of choosing $x_t\in\mathcal{X}_{S[t]}$ as a stationary and randomized function of $S[t]$, i.e. has the same conditional distribution given that the same $S[t]$ value is observed, and independent of $\mathcal{H}_t$.
\end{definition}

Let $\Gamma^*$ be the set of all ``one-shot'' expectations $\expect{x_t}\in\mathbb R^d$ that are possible at any given time slot $t$, considering all possible randomized stationary algorithms for choosing $x_t\in\mathcal X_{S[t]}$ in reaction to the observed $S[t]$. Since $\mathcal X_{S[t]}\subseteq \mathcal{B}$, it follows $\Gamma^*$ is a bounded set.
Define $\overline\Gamma^*$ as the closure of $\Gamma^*$. 
It can be shown that both $\Gamma^*$ and $\overline\Gamma^*$ are convex 
subsets of $\mathcal{B}$ (\cite{neely2010stochastic}).  Let $f:\mathcal{B}\rightarrow\mathbb{R}$ be a given 
cost function and $\mathbf{a}_i\in\mathbb{R}^d,~i=1,2,\cdots,N$ be a given collection of constraint vectors.  
The goal is to make sequential decisions that leads to the construction of a vector $\gamma \in \mathbb{R}^d$ that solves
the following optimization problem:
\begin{align}
\min~~&f(\gamma)    \label{prob-1}\\
s.t.~~&\gamma\in \overline\Gamma^* \label{prob-2}\\
&\dotp{\mathbf{a}_i}{\gamma}\leq b_i,~\forall i\in\{1,2,\cdots,N\} \label{prob-3}
\end{align}
where $b_i,~i=1,2,\cdots,N$ are given constants, and $\dotp{\mathbf{a}_i}{\gamma}$ denotes the dot product of vectors $\gamma$ and $\mathbf{a}_i$.  The constraint \eqref{prob-3} is equivalent to $\mathbf{A}\gamma \leq \mathbf{b}$
with definitions $\mathbf{A}:=[\mathbf{a}_1,\mathbf{a}_2,\cdots,\mathbf{a}_N]^T$ and 
$\mathbf{b}:=[b_1,b_2,\cdots,b_N]^T$. 

The problem of producing a vector $\gamma$ that solves 
\eqref{prob-1}-\eqref{prob-3} is particularly challenging when the probability distribution for $S[t]$ is unknown. 
Algorithms that make decisions without knowledge of this distribution are called \emph{statistics-unaware} algorithms
and are the focus of the current paper. The first part of this paper treats convex functions $f$ and 
the vector $\gamma$ shall correspond to 
the actual time averages achieved by the decisions $x_t$ made over time.  This part considers 
convergence to an $\varepsilon$-approximation of the global optimal solution to  \eqref{prob-1}-\eqref{prob-3}, for arbitrary
$\varepsilon>0$. The second part of this paper treats non-convex  
$f$ and generates a vector $\gamma$ that is informed by the $x_t$ decisions but is not the time average. 
This non-convex analysis does not treat convergence to a global optimum. Rather, it 
considers convergence to a vector that makes a local sub-optimality gap small.   

Such a program \eqref{prob-1}-\eqref{prob-3} 
appears in a number of scenarios including wireless scheduling systems and power optimization in a
network subject to queue stability (e.g. \cite{agrawal2002optimality, prop-fair-down, andrews2005optimal, eryilmaz2007fair, lee2006opportunistic, stolyar2005asymptotic, stolyar2005maximizing}).  
For example, consider a wireless system with $d$ users that transmit over their own links. The wireless channels can change over time and this affects the set of transmission rates available
for scheduling. Specifically, the random sequence $\l\{ S[t] \r\}_{t=0}^\infty$ can be a process of independent and identically distributed (i.i.d.) channel state vectors that take values in some set $\mathcal{S}\subseteq \mb R^d$. The decision variable $x_t$ is the transmission rate vector chosen from $\Gamma_{S[t]}$, which is the set of available rate vectors determined by the observation $S[t]$, and $\overline\Gamma^*$ is often referred to as the \emph{capacity region} of the network.  The function $f(\cdot)$ often represents the negative network utility function, and the constraints \eqref{prob-3} represent additional system requirements 
or resource restrictions, such as constraints on average power consumption, requirements on throughput and fairness, etc.  Such a problem is called \textit{opportunistic scheduling} because the network controller can observe the current channel state vector and can choose to transmit over certain links whenever their channel states are good.

\subsection{Stochastic optimization with convex objectives}
A typical algorithm solving (\ref{prob-1}-\ref{prob-3}) is the \textit{drift-plus-penalty (DPP) algorithm} (e.g. \cite{neely2010stochastic, neely2008fairness}). It is shown that when the function $f(\cdot)$ is convex,  this algorithm achieves an $\varepsilon$ approximation with the convergence time $\mathcal{O}(1/\varepsilon^2)$ under mild assumptions. The algorithm features a ``virtual queue'' $Q_i(t)$ for each constraint, which is 0 at $t=0$ and updated as follows:
\begin{equation}\label{eq:queue-update}
Q_i(t+1) = \max\l\{Q_i(t) + \dotp{\mathbf{a}_i}{x_t} - b_i,0\r\},~i=1,2,\cdots,N.
\end{equation}
Then, at each time slot, the system chooses $x_t\in\mathcal{X}_{S[t]}$ after observing $S[t]$ as follows:
\begin{equation}\label{eq:dpp}
x_t := \argmin_{x\in\mathcal{X}_{S[t]}} \left[\frac{1}{\varepsilon}f(x) + \sum_{i=1}^NQ_i(t)\dotp{\mathbf{a}_i}{x}\right].
\end{equation}
This algorithm does not require knowledge of the distribution of $S[t]$ (it is \emph{statistics-unaware}) 
and does not require $f(\cdot)$ to be differentiable or smooth. However, it requires one to have  full knowledge of the function 
$f(\cdot)$ and to solve \eqref{eq:dpp} efficiently at each time slot.

Another body of work considers primal-dual gradient based methods solving (\ref{prob-1}-\ref{prob-3}), where, instead of solving the subproblem like \eqref{eq:dpp} slot-wise, they solve a subproblem which involves only the gradient of $f(\cdot)$ at the current time step (e.g. \cite{agrawal2002optimality}\cite{prop-fair-down} for unconstrained problems and \cite{stolyar2005asymptotic, stolyar2005maximizing} for constrained problems). Specifically, assume the function $f$ is differentiable and define
the gradient at a point $x\in\mathcal{B}$ by 
\[
\nabla f(x) := \l[\frac{\partial f(x)}{\partial x_1},~\frac{\partial f(x)}{\partial x_2},~\cdots,~\frac{\partial f(x)}{\partial x_d} \r]^T,
\]
where $x_j$ is the $j$-th entry of $x\in\mb R^d$. Then, introduce at parameter $\beta>0$ and a queue variable $Q_i(t)$ for each constraint $i$, which is updated in the same way as \eqref{eq:queue-update}. Then, during each time slot, the system solves the following problem: $\gamma_{-1} = 0$,
\begin{equation}\label{gradient-update}
x_t := \argmin_{x\in\mathcal{X}_{S[t]}}  \l( \dotp{\nabla f(\gamma_{t-1})}{x} + \beta\sum_{i=1}^NQ_i(t)\dotp{\mathbf{a}_i}{x}\r),
\end{equation}
and $\gamma_t = (1-\beta)\gamma_{t-1} + \beta x_t$. This method requires $f(\cdot)$ to be smooth and convex.
It is analyzed in \cite{stolyar2005asymptotic, stolyar2005maximizing} via a control system approach approximating the dynamics of $(x_t,Q_1(t),\cdots,Q_N(t))$ by the trajectory of a continuous dynamical system (as $\beta\rightarrow0$). Such an argument makes precise convergence time bounds difficult to obtain. To the best of our knowledge, there is no formal convergence time analysis of this primal-dual gradient algorithm.

More recently, the work \cite{neely2017frank} analyzes gradient type algorithms 
solving the stochastic problem without linear constraints:
\begin{equation}\label{eq:ppp}
\min~~f(\gamma)  ~~
s.t.~~\gamma\in \overline\Gamma^*, 
\end{equation}
where $\overline\Gamma^*$ has the same definition as \eqref{prob-2}. The algorithm starts with $\gamma_{-1} = 0$ and solves the following linear optimization each time slot,
\[
x_t := \argmin_{x\in\mathcal{X}_{S[t]}}  \dotp{\nabla f(\gamma_{t-1})}{x},
\]
then, the algorithm updates $\gamma_t = (1-\beta_t)\gamma_{t-1} + \beta_t x_t$, where $\beta_t$ is a sequence of pre-determined weights. The work \cite{neely2017frank} shows that the online time averages of this 
algorithm have precise $\mathcal{O}(1/\varepsilon^2)$ and $\mathcal{O}(\log(1/\varepsilon)/\varepsilon)$ convergence time bounds reaching an 
$\varepsilon$ approximation, depending on the choices of $\beta_t$.
When the set $\mathcal{X}_{S[t]}$ is fixed and does not change with time, this algorithm reduces to  the Frank-Wolfe algorithm for deterministic convex minimization (e.g. \cite{bubeck2015convex, jaggi2013revisiting, nesterov2015complexity}), 
which serves as a starting point of this work.

\subsection{Optimization with non-convex objectives}
While convex stochastic optimization has been studied extensively and various algorithms have been developed, much less is known when $f(\cdot)$ is non-convex. 
On the other hand, non-convex problems have many applications in network utility maximization and in other areas of network science. For example, suppose we desire to maximize a non-concave utility function of throughput in a wireless network. An example is the
``sigmoidal-type'' utility function (Fig. \ref{ss}), which has flat regions both near the origin and when the attribute is large, characterizing a typical scenario where we earn small
utility unless throughput crosses a threshold and the utility does not increase too much if we keep increasing the throughput. 

 \begin{figure}[htbp]
 \centering
   \includegraphics[width=3in]{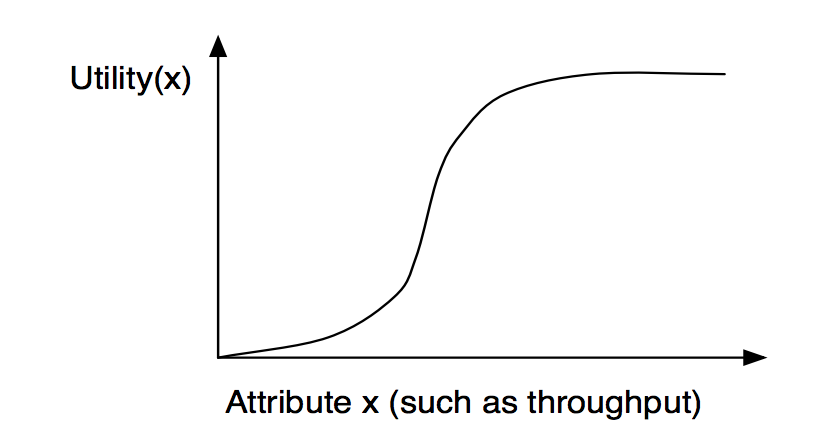} 
   \caption{A sigmoidal-type utility function} 
   \label{ss}
\end{figure}

Constrained non-convex stochastic optimization with a sum of such utility functions is known to be difficult to solve and heuristic algorithms have been proposed in several works to solve this problem (e.g. \cite{chiang2009nonconvex, lee2005non}). In \cite{chiang2009nonconvex}, the author proposes a convex relaxation heuristic treating the deterministic scenario where $\mathcal{X}_{S[t]}$ is fixed. The work \cite{lee2005non} considers the static network scheduling problem and develops a method which has optimality properties in a limit of a large number of users. The work \cite{neely2010stochastic-nonconvex} develops a primal-dual type algorithm which can be shown to achieve a local optimal of (\ref{prob-1}-\ref{prob-3}) when a certain sequence generated by the algorithm itself is assumed to converge. However, whether or not such a sequence converges is not known.

More recently, non-convex optimization has regained attention mainly due to the need for  training deep neural networks and other machine learning models, and a thorough review of relevant literature is beyond the scope of this paper. Here, we only highlight several works related to the formulation (\ref{prob-1}-\ref{prob-3}). The works \cite{wang2015global, yang2017alternating} show that ADMM type algorithms converge when minimizing a non-convex objective (which can be written in multiple blocks) subject to linear constraints. The work \cite{li2016douglas} shows the Douglas-Rachford operator splitting method converges when solving non-convex composite optimization problems, which includes the linear constrained non-convex problem as a special case. On the other hand, Frank-Wolfe type algorithms have also been applied to solve non-convex constrained problems. The work \cite{lacoste2016convergence} considers using Frank-Wolfe to solve problems of the form 
$$\min_{x\in \mathcal{M}}f(x),$$ 
where $f(\cdot)$ is possibly non-convex, and shows that the ``Frank-Wolfe gap'', which measures the local suboptimality, converges on the order of  
$1/\sqrt{T}$ when running for $T$ slots. The work \cite{reddi2016stochastic} generalizes the previous results to solve 
$$\min_{x\in \mathcal{M}}\mathbb{E}_{\theta}F(x,\theta),$$
where the expectation is taken with respect to the random variable $\theta$. Note that this problem is fundamentally different from \eqref{eq:ppp} because the above problem aims at choosing a \textit{fixed} $x\in\mathcal{M}$ to minimize the expectation of a function, whereas the problem \eqref{eq:ppp} aims at choosing a \textit{policy} $x_t\in\mathcal{X}_{S[t]}$, in reaction to $S[t]$ at each time slot, whose expectation $\gamma$ minimizes $f(\gamma)$.

\subsection{Contributions}
In this work, we propose a primal-dual Frank-Wolfe algorithm solving (\ref{prob-1}-\ref{prob-3}) with both convex and nonconvex objectives. Specifically,
\begin{itemize}
\item When the objective is convex, we show that the proposed algorithm gives a convergence time of $\mathcal{O}(1/\varepsilon^3)$. Further, we show an improved convergence time of 
$\mathcal{O}(1/\varepsilon^2)$ holds under a mild assumption on existence of a Lagrange multiplier vector. 
Such rates tie with the best known convergence time for (\ref{prob-1}-\ref{prob-3}) achieved by the Drift-plus-penalty algorithm \eqref{eq:dpp} but with possibly lower complexity update per slot, because we only require minimizing the linear approximation of $f(\cdot)$ rather than $f(\cdot)$ itself. Furthermore, this result also implies a precise convergence time guarantee for the primal dual gradient algorithm \eqref{gradient-update} in earlier works.
\item When the objective is non-convex, we show our proposed algorithm converges to the local minimum with $\mathcal{O}(1/\varepsilon^3)$ convergence time on the ``Frank-Wolfe'' gap, and an improved $\mathcal{O}(1/\varepsilon^2)$ convergence time when Slater's condition holds. To the best of authors' knowledge, this is the first algorithm that treats (\ref{prob-1}-\ref{prob-3}) 
for non-convex objectives with provable convergence guarantees on the ``Frank-Wolfe'' gap. 
We also emphasize the application of our proposed algorithm to non-convex opportunistic scheduling and non-convex distributed stochastic optimization.
\end{itemize}

\section{Algorithm}
In this section, we introduce a primal-dual type algorithm that works for both convex and non-convex cost functions $f(\cdot)$.
Throughout the paper, we assume that the set $\overline{\Gamma}^*$ contains the origin, $f:\mathcal{B}\rightarrow\mathbb{R}$ is a smooth function with $L$-Lipschitz gradient, i.e.
\[
\|\nabla f(x) - \nabla f(y)\|\leq L\|x-y\|,~\forall x,y\in \mathcal{B},
\]  
where for any $x\in\mathbb{R}^d$, $\|x\|:=\sqrt{\sum_{i=1}^dx_i^2}$.  Furthermore, we assume the following quantities are bounded
by some constants $M, K, B, D$: 
\begin{align}
&\sup_{\gamma\in\mathcal{B}}\|\nabla f(\gamma)\|\leq M, \label{eq:M} \\
&\sup_{\gamma\in\mathcal{B}}|f(\gamma)|\leq K,\label{eq:K} \\
&\sup_{\gamma\in\mathcal{B}}\sum_{i=1}^N|\dotp{\mathbf{a}_i}{\gamma} - b_i|^2\leq B^2,~\forall i\in\{1,2,\cdots,N\}, \label{eq:B} \\
&\|x-y\|\leq D,~\forall x,y\in\mathcal{B}.   \label{eq:D} 
\end{align}

The proposed algorithm is as follows: For a time horizon $T>1$, let $\eta,V$ be two algorithm parameters that shall affect a performance tradeoff.  Assume throughout that $0 < \eta < 1$.  Let 
$\mathbf Q(t) = [Q_1(t),~Q_2(t),\cdots,~Q_N(t)]$ be a vector of virtual queues with $Q_i(0) = 0,~i=1,2,\cdots,N$. These virtual queues can be viewed as penalizations to the violations of the constraints. 
Let $\gamma_{-1} = 0\in\mathbb{R}^d$ 
and at each time slot $t\in\{0,1,2,\cdots,T\}$, do the following:\footnote{Since $0 \in \overline{\Gamma}^*$   and $\overline{\Gamma}^*\subseteq \mathcal{B}$, it
holds that $\gamma_{-1}=0 \in \mathcal{B}$.}

\begin{enumerate}
\item Observe $S[t]$ and choose $x_t\in\mathcal{X}_{S[t]}$ to solve the following linear optimization problem:
\begin{equation}\label{sub-problem}
\min_{x\in\mathcal{X}_{S[t]}}   \l( V\dotp{\nabla f(\gamma_{t-1})}{x} + \sum_{i=1}^NQ_i(t)\dotp{\mathbf{a}_i}{x}\r).
\end{equation}
\item Update $\gamma_t$ via
$
\gamma_t = (1-\eta)\gamma_{t-1}+\eta x_t.
$ 
Update the virtual queue via 
\begin{equation*}
Q_i(t+1) = \max\l\{Q_i(t) + \dotp{\mathbf{a}_i}{x_t} - b_i,0\r\},~i=1,2,\cdots,N.
\end{equation*}
\end{enumerate}

Since $\mathcal{X}_{S[t]}$ is assumed to be a compact set, there always exists at least one solution to the optimization problem \eqref{sub-problem}. Furthermore, this algorithm is similar to the primal dual gradient type algorithm \eqref{gradient-update}. Specifically, when choosing $V= 1/\beta$ and $\eta = \beta$,  \eqref{sub-problem} is the same as \eqref{gradient-update}. Thus, convergence analysis of this algorithm can be applied to the primal-dual gradient algorithm as a special case. But as we shall see, having different weights in the algorithm sometimes helps speed up the convergence. 

\section{Performance bounds for convex cost functions}

In this section, we analyze the convergence time of the proposed algorithm when the cost function $f(\cdot)$ is smooth and 
convex. We focus on the performance of the \textbf{time average solution
$\overline x_T := \frac{1}{T}\sum_{t=0}^{T-1}x_t$}.

\subsection{Preliminaries}
We start with the definition of \emph{$\varepsilon$-near optimality}. 
\begin{definition}
Let $\gamma^*$ be the solution to (\ref{prob-1}-\ref{prob-3}). 
Given an accuracy level $\varepsilon>0$, a point $\gamma\in\overline{\Gamma}^*$ is said to be an $\varepsilon$-near optimal solution   to 
(\ref{prob-1}-\ref{prob-3}) if it satisfies
\[
f(\gamma)\leq f(\gamma^*)+\varepsilon,~~\dotp{a_i}{\gamma}-b_i\leq\varepsilon,~\forall i\in\{1,2,\cdots,N\}.
\]
\end{definition}

Next, we review some basic properties of smooth convex functions. it is known that every convex and differentiable function $f:\mathcal{B}\rightarrow \mathbb R$ satisfies the following inequality:
\begin{equation}\label{eq:convexity}
f(y) \geq f(x) + \dotp{\nabla f(x)}{y-x},~\forall x,y\in\mathcal{B}.
\end{equation}
Furthermore, every $L$-smooth function $f:\mathcal{B}\rightarrow \mathbb{R}$ satisfies the following inequality, often called the \emph{descent lemma}:
\begin{equation}\label{eq:smooth}
f(y) \leq f(x) + \dotp{\nabla f(x)}{y-x} + \frac L2\|y-x\|^2
\end{equation}


\begin{lemma}\label{lem:key-equation}
For any smooth cost  
function $f:\mathcal{B}\rightarrow \mathbb{R}$ and every time slot $t\in\{0,1,2,\cdots\}$, the proposed algorithm guarantees that
\begin{equation}\label{eq:key-equation}
V\expect{\dotp{\nabla f(\gamma_{t-1})}{x_t-\gamma_{t-1}} | \mathcal{H}_t} + \expect{\dotp{\mathbf Q(t)}{\mathbf{A}x_t-\mathbf{b}}|\mathcal{H}_t}\leq 
V\dotp{\nabla f(\gamma_{t-1})}{ \gamma^*-\gamma_{t-1}},
\end{equation}
where $\gamma^*$ is any point satisfying the constraints (\ref{prob-2}-\ref{prob-3}).
\end{lemma}
\begin{proof}[Proof of Lemma \ref{lem:key-equation}]
For any time slot $t$, since $x_t\in\mathcal{X}_{S[t]}$ solves \eqref{sub-problem}, the following holds:
\[
V\dotp{\nabla f(\gamma_{t-1})}{x_t-\gamma_{t-1}} + \dotp{\mathbf Q(t)}{\mathbf{A}x_t-\mathbf{b}}
\leq V \dotp{\nabla f(\gamma_{t-1})}{v_t-\gamma_{t-1}} + \dotp{\mathbf Q(t)}{\mathbf{A}v_t-\mathbf{b}},
\]
where $v_t\in\mathcal{X}_{S[t]}$ is a vector selected by any randomized stationary algorithm. As a consequence, any vector $\gamma\in\Gamma^*$ satisfying $\dotp{\mathbf{a}_i}{\gamma}\leq b_i,~i=1,2,\cdots,N$ can be achieved by some $v_t$ such that $\gamma = \expect{v_t}$. Taking conditional expectations from both sides of the above inequality, conditioning on $\mathcal{H}_t$, we have 
\begin{align*}
&V \expect{\dotp{\nabla f(\gamma_{t-1})}{x_t-\gamma_{t-1}} | \mathcal{H}_t} + \expect{\dotp{\mathbf Q(t)}{\mathbf{A}x_t-\mathbf{b}} | \mathcal{H}_t}\\
&\leq V  \expect{\dotp{\nabla f(\gamma_{t-1})}{v_t-\gamma_{t-1}} | \mathcal{H}_t} + \expect{\dotp{\mathbf Q(t)}{\mathbf{A}v_t-\mathbf{b}} | \mathcal{H}_t}\\
&\overset{(a)}{=} V \dotp{\nabla f(\gamma_{t-1})}{ \expect{v_t| \mathcal{H}_t}-\gamma_{t-1}} + \dotp{\mathbf Q(t)}{\expect{\mathbf{A}v_t-\mathbf{b} | \mathcal{H}_t}- \mathbf{b}} \\
&\overset{(b)}{=} V \dotp{\nabla f(\gamma_{t-1})}{ \gamma-\gamma_{t-1}} + \dotp{\mathbf Q(t)}{\mathbf{A}\gamma-\mathbf{b}}\\
&\overset{(c)}{\leq} V \dotp{\nabla f(\gamma_{t-1})}{ \gamma-\gamma_{t-1}},
\end{align*}
where (a) follows from the fact that $\gamma_{t-1}$ and $\mathbf Q(t)$ are determined by $\mathcal{H}_t$; (b)  follows from the fact that $v_t\in\mathcal{X}_{S[t]}$ is generated by a randomized stationary algorithm and independent of the system history $\mathcal{H}_t$, which implies $\expect{v_t| \mathcal{H}_t} = \expect{v_t} = \gamma$; (c) follows from the assumption that $\gamma$ satisfies the constraint 
$\mathbf{A}\gamma\leq\mathbf{b}$ and $\mathbf{Q}(t)\geq0$, where the inequalities are taken to be entrywise.
Note that the preceding inequality holds for any $\gamma\in\Gamma^*$ satisfying $\mathbf{A}\gamma\leq\mathbf{b}$. Taking a limit as $\gamma\rightarrow \gamma^*$ for any $\gamma^*$ satisfying the constraints (\ref{prob-2}-\ref{prob-3}) gives the result.
\end{proof}

\begin{lemma}\label{lem:key-equation-2}
For any smooth convex cost function $f:\mathcal{B}\rightarrow \mathbb{R}$ and every time slot $t\in\{0,1,2,\cdots\}$, the proposed algorithm guarantees that
\begin{equation}\label{eq:key-equation-2}
V \expect{\dotp{\nabla f(\gamma_{t-1})}{x_t-\gamma_{t-1}}} + \expect{\dotp{\mathbf Q(t)}{\mathbf{A}x_t-\mathbf{b}}}\leq 
V (f(\gamma^*) - \expect{f(\gamma_{t-1})}),
\end{equation}
where $\gamma^*$ is any point satisfying the constraints (\ref{prob-2}-\ref{prob-3}).
\end{lemma}
\begin{proof}[Proof of Lemma \ref{lem:key-equation-2}]
By the inequality \eqref{eq:convexity} for convex functions, we have
\[
\dotp{\nabla f(\gamma_{t-1})}{  \gamma^* - \gamma_{t-1}}
\leq f(\gamma^*) - f(\gamma_{t-1}).
\]
Substituting this inequality into  \eqref{eq:key-equation} gives
\[
V \expect{\dotp{\nabla f(\gamma_{t-1})}{x_t-\gamma_{t-1}} | \mathcal{H}_t} + \expect{\dotp{\mathbf Q(t)}{\mathbf{A}x_t-\mathbf{b}}|\mathcal{H}_t}\leq
V  (f(\gamma^*) - f(\gamma_{t-1}))
\]
Taking expectations from both sides of the above inequality gives the result.
\end{proof}

\begin{lemma}\label{lem:q-constraint}
Let $\overline x_T := \frac{1}{T}\sum_{t=0}^{T-1}x_t$. Then, we have for any $i\in\{1,2,\cdots,N\}$,
\[
\dotp{\mathbf{a}_i}{\expect{\overline{x}_T}}-b_i\leq \frac{\expect{Q_i(T)}}{T}.
\]
\end{lemma}
\begin{proof}[Proof of Lemma \ref{lem:q-constraint}]
By virtual queue updating rule, we have
\[
Q_i(t+1)= \max\l\{ Q_i(t) + \dotp{\mathbf{a}_i}{x_t} - b_i,~0 \r\}\geq Q_i(t) + \dotp{\mathbf{a}_i}{x_t} - b_i.
\]
Summing over $t\in\{0,1,2,\cdots,T-1\}$ gives
\[
Q_i(T) \geq Q_i(0) + \sum_{t=0}^{T-1}\l(\dotp{\mathbf{a}_i}{x_t} - b_i\r) = \sum_{t=0}^{T-1}\l(\dotp{\mathbf{a}_i}{x_t} - b_i\r)
\]
which uses $Q_i(0)=0$. 
Taking expectations and dividing by $T$ gives the inequality. 
\end{proof} 

\begin{lemma}\label{lem:diff}
Let $\overline x_T := \frac{1}{T}\sum_{t=0}^{T-1}x_t$. Then 
\[
 \overline{x}_T  - \mbox{$\frac{1}{T}\sum_{t=0}^{T-1}\gamma_{t-1}$}= \frac{\gamma_{T-1}}{\eta T},
\]
and thus 
$$\|\overline{x}_{T} - \mbox{$\frac{1}{T}\sum_{t=0}^{T-1} \gamma_{t-1}$}\| \leq\frac{D}{\eta T}$$
where $D$ is defined in \eqref{eq:D}. 
\end{lemma}
\begin{proof}[Proof of Lemma \ref{lem:diff}]
Note that 
$\gamma_t = (1-\eta)\gamma_{t-1} + \eta x_t$ for all $t \in \{0, 1, 2, \ldots\}$. Thus, 
\begin{equation*}
x_t-\gamma_{t-1} = \frac{\gamma_t-\gamma_{t-1}}{\eta}
\end{equation*}
Summing the above equality over $t \in \{0, \ldots, T-1\}$ yields 
$$  \sum_{t=0}^{T-1} (x_t-\gamma_{t-1}) = \frac{\gamma_{T-1}-\gamma_{-1}}{\eta} $$
Dividing the above equality by $T$ and using $\gamma_{-1}=0$ gives
$$  \overline{x}_{T} -\frac{1}{T}\sum_{t=0}^{T-1} \gamma_{t-1}= \frac{\gamma_{T-1}}{\eta T}$$
Thus, 
\[
\|\overline x_T - \mbox{$\frac{1}{T}\sum_{t=0}^{T-1} \gamma_{t-1}$}\| = \frac{1}{\eta T}\|\gamma_{T-1}\|\leq\frac{D}{\eta T},
\]
where the final inequality uses the fact that  $\gamma_{T-1} \in \mathcal{B}$ and $0 \in \mathcal{B}$ and 
so the boundedness assumption \eqref{eq:D} ensures  $\|\gamma_{T-1}-0\| \leq D$. 
\end{proof}

\subsection{Convergence time bound}

Define the drift $\Delta(t) := \frac12\l( \|\mathbf Q(t+1)\|^2 - \|\mathbf Q(t)\|^2 \r)$. 
We have the following cost and virtual queue bounds:
\begin{theorem}\label{thm:convex-1}
For any $T>1$, and any $\gamma^*\in\overline\Gamma^*$ we have
\begin{align*}
&\expect{f(\overline{x}_T)}\leq f(\gamma^*) + \frac{2K+MD}{\eta T} + \frac{B^2}{2V} + \frac{LD^2\eta}{2} ,\\
&\frac{\expect{\|\mathbf Q(T)\|^2}}{T^2}\leq \frac{4KV}{T} + \frac{4KV}{T^2\eta} +\frac{B^2}{T} +\frac{LD^2V\eta}{T}.
\end{align*}
and by Jensen's inequality the objective function bound also holds for $f(\expect{\overline{x}_{T}})$. 
\end{theorem}

\begin{proof}[Proof of Theorem \ref{thm:convex-1}]

First of all, using the smoothness of the cost function $f(\cdot)$, for any $t\in\{0,1,2,\cdots\}$,
\[
f(\gamma_t)\leq f(\gamma_{t-1}) + \dotp{\nabla f(\gamma_{t-1})}{\gamma_t-\gamma_{t-1}} + \frac{L}{2}\|\gamma_t-\gamma_{t-1}\|^2.
\]
Substituting $\gamma_t- \gamma_{t-1} = \eta(x_t-\gamma_{t-1})$  
into the above gives
\[
f(\gamma_t)\leq f(\gamma_{t-1}) + \eta\dotp{\nabla f(\gamma_{t-1})}{x_t-\gamma_{t-1}} + \frac{L\eta^2}{2}\|x_t - \gamma_{t-1}\|^2
\]
Multiplying by $V$, adding $\eta\Delta(t)$ to both sides and taking expectations gives 
\begin{align}
&V\expect{f(\gamma_t)} + \eta\expect{\Delta(t)}   \nonumber\\
&\leq V\expect{f(\gamma_{t-1})} + V\eta\expect{\dotp{\nabla f(\gamma_{t-1})}{x_t-\gamma_{t-1}}} + \frac{LV\eta^2}{2}\expect{\|x_t-\gamma_{t-1}\|^2} + \eta \expect{\Delta(t)}  \nonumber\\
&\overset{(a)}{\leq}  V\expect{f(\gamma_{t-1})} + V\eta\expect{\dotp{\nabla f(\gamma_{t-1})}{x_t-\gamma_{t-1}}} + \frac{LD^2V\eta^2}{2} 
+ \frac{\eta B^2}{2} + \eta\expect{\dotp{\mathbf Q(t)}{\mathbf{A}x_t-\mathbf{b}}}   \nonumber\\
&\overset{(b)}{\leq} V\expect{f(\gamma_{t-1})}+ V\eta(f(\gamma^*) - \expect{f(\gamma_{t-1})}) + \frac{LD^2V\eta^2+\eta B^2}{2}, \label{eq:inter-2}
\end{align}
where (a) follows from $\|x_t - \gamma_{t-1}\|^2\leq D^2$  (by \eqref{eq:D}), and
\begin{multline*}
\Delta(t) = \frac12\l( \|\mathbf Q(t+1)\|^2 - \|\mathbf Q(t)\|^2 \r)
\leq\frac12\l( \sum_{i=1}^N\l(Q_i(t) + \dotp{\mathbf{a}_i}{x_t} - b_i \r)^2 - \sum_{i=1}^NQ_i(t)^2\r)\\
= \frac12\sum_{i=1}^N\l( \dotp{\mathbf{a}_i}{x_t} - b_i \r)^2 + \sum_{i=1}^NQ_i(t)\l( \dotp{\mathbf{a}_i}{x_t} - b_i \r)
\leq  \frac{B^2}{2} + \dotp{\mathbf Q(t)}{\mathbf{A}x_t-\mathbf{b}},
\end{multline*}
with the final inequality holding by the $B^2$ bound \eqref{eq:B}; while (b) follows from Lemma \ref{lem:key-equation-2}. Thus, \eqref{eq:inter-2} implies
\begin{equation}\label{eq:inter-3}
\expect{f(\gamma_{t-1})}\leq f(\gamma^*) - \frac{\expect{f(\gamma_t)}}{\eta}+\frac{\expect{f(\gamma_{t-1})}}{\eta} - \frac{\expect{\Delta(t)}}{V} 
+ \frac{B^2}{2V} + \frac{LD^2}{2}\eta,
\end{equation}
Summing both sides over  $t\in \{0, \ldots, T-1\}$ and dividing by $T$ gives 
\begin{align*}
\frac1T\sum_{t=0}^{T-1}\expect{f(\gamma_{t-1})}\leq& f(\gamma^*) - \frac{\expect{f(\gamma_{T-1})}}{\eta T}+\frac{\expect{f(\gamma_{-1})}}{\eta T} 
- \frac{\expect{\|\mathbf Q(T)\|^2-\|\mathbf Q(0)\|^2}}{2VT} 
+ \frac{B^2}{2V} + \frac{LD^2}{2}\eta\\
\overset{(a)}{\leq}& f(\gamma^*) + \frac{2K}{\eta T} + \frac{B^2}{2V} + \frac{LD^2}{2}\eta
\end{align*}
where (a) follows from the $K$ bound in \eqref{eq:K} and from $\mathbf Q(0)=0$. 
Applying Jensen's inequality gives
\begin{equation}\label{eq:inter-1.1}
\expect{f\left(\underbrace{\mbox{$\frac{1}{T}\sum_{t=0}^{T-1}\gamma_{t-1}$}}_{\theta_T}\right)}\leq f(\gamma^*) + \frac{2K}{\eta T} + \frac{B^2}{2V} + \frac{LD^2}{2}\eta.
\end{equation}
where $\theta_T = \frac{1}{T}\sum_{t=0}^{T-1}\gamma_{t-1}$ is defined for convenience. 
By \eqref{eq:convexity},
\begin{align*}
f(\theta_T)\geq& f(\overline x_{T}) + \dotp{\nabla f(\overline x_{T})}{\theta_T - \overline x_{T}} \\
\geq& f(\overline x_{T}) - \|\nabla f(\overline x_{T})\| \cdot\|\theta_T - \overline x_{T}\|\\
\geq&f(\overline x_{T}) - \frac{MD}{\eta T} ,
\end{align*}
where the last inequality follows from the $M$ bound in \eqref{eq:M} 
and the fact that $\| \theta_T - \overline{x}_T\| \leq D/(\eta T)$ by Lemma \ref{lem:diff}. Overall, we have
\[  f(\overline x_{T}) \leq f(\theta_T) + \frac{MD}{\eta T} 
\]
Substituting this bound into \eqref{eq:inter-1.1} completes the proof of the objective bound.

To get the $\mathbf{Q}(T)$ bound, we sum  \eqref{eq:inter-3} over $t \in \{0, \ldots, T-1\}$ to obtain 
\begin{equation}\label{eq:inter-last}
\sum_{t=-1}^{T-2}\expect{f(\gamma_{t})}\leq Tf(\gamma^*) - \frac{\expect{f(\gamma_{T-1})}}{\eta }+\frac{\expect{f(\gamma_{-1})}}{\eta } 
- \frac{\expect{\|\mathbf Q(T)\|^2-\|\mathbf Q(0)\|^2}}{2V } 
+ \frac{B^2T}{2V} + \frac{LD^2}{2}\eta T.
\end{equation}
Rearranging the terms with the fact that $\mathbf Q(0) = 0$ gives
\begin{align*}
\expect{\|\mathbf Q(T)\|^2} \leq& 2V\l(  Tf(\gamma^*) - \sum_{t=-1}^{T-2}\expect{f(\gamma_{t})} \r) + \frac{2V}{\eta}\l(\expect{f(\gamma_{-1})} - \expect{f(\gamma_{T-1})}\r) + B^2T + LD^2\eta VT\\
\leq& 4KVT + \frac{4K V}{\eta} + B^2 T + LD^2\eta VT,
\end{align*}
which implies the claim dividing $T^2$ from both sides.
\end{proof}

\begin{corollary}\label{col:convex-1}
Let $V=T^{1/3}$ and $\eta=\frac{1}{T^{2/3}}$, then, we have objective and constraint violation bounds as follows,
\begin{align*}
&\expect{f(\overline{x}_T)}\leq f(\gamma^*) + \l( 2K + MD + \frac{B^2}{2} + \frac{LD^2}{2T^{1/3}}\r)\frac{1}{T^{1/3}},\\
&\dotp{\mathbf{a}_i}{\expect{\overline{x}_T}} - b_i \leq \left[\sqrt{4K+ \frac{4K}{T^{1/3}} + \frac{B^2}{T^{1/3}} + \frac{LD^2}{T^{2/3}}}\right]\frac{1}{T^{1/3}},~\forall i\in\{1,2,\cdots,N\}.
\end{align*}
and by Jensen's inequality the objective function bound also holds for $f(\expect{\overline{x}_{T}})$. 
Thus, to achieve an $\varepsilon$-near optimality, the convergence time is $\mathcal{O}(1/\varepsilon^3)$.
\end{corollary}
\begin{proof}
The objective bound follows directly from Theorem \ref{thm:convex-1}. The constraint violation bound follows from taking the square root of the virtual queue bound in Theorem \ref{thm:convex-1} and using Lemma \ref{lem:q-constraint}.
\end{proof}

\subsection{Convergence time improvement via Lagrange multipliers}

In this section, we show the objective and constraint violation bounds in Corollary \ref{col:convex-1} can be improved from $\mathcal{O}(1/\varepsilon^3)$ to $\mathcal{O}(1/\varepsilon^2)$ when a Lagrange multiplier vector exists.
 
\begin{assumption}\label{assume:Lagrange}
There exists a non-negative vector $\lambda\in \mathbb{R}^N$ such that
\[
f(\gamma) + \dotp{\lambda}{\mathbf{A}\gamma-\mathbf{b}}\geq f(\gamma^*), ~~\forall \gamma\in\overline{\Gamma}^*
\]
where $\gamma^*$ is the solution to the optimization problem (\ref{prob-1}-\ref{prob-3}).
\end{assumption}
It can be shown that such an assumption is mild and, in particular, it is true when a Slater  condition holds (see Assumption \ref{ass:slater} for specifics).

\begin{lemma}\label{thm:convex-2}
Let $\gamma^*\in\overline\Gamma^*$ be the solution to the optimization problem (\ref{prob-1}-\ref{prob-3}). For any $T>0$, 
\begin{align*}
&\expect{f(\overline{x}_T)}\leq f(\gamma^*) + \frac{2K+MD}{\eta T} + \frac{B^2}{2V} + \frac{LD^2\eta}{2},\\
&\expect{\|\mathbf{Q}(T)\|} \leq 2\|\lambda\| V + \sqrt{\frac{2\|\mathbf A^T\lambda\|DV+ 4KV}{\eta}+B^2T + LD^2\eta V T}
\end{align*}
\end{lemma}

\begin{proof}[Proof of Lemma \ref{thm:convex-2}]
The objective bound follows directly from that of Theorem \ref{thm:convex-1}. To get the constraint violation bound, note that 
$\frac{1}{T}\sum_{t=-1}^{T-2}\expect{\gamma_t}\in\overline{\Gamma}^*$. By Assumption \ref{assume:Lagrange}, we have
\[
f\l(\frac{1}{T}\sum_{t=-1}^{T-2}\expect{\gamma_t}\r) + \dotp{\lambda}{\mathbf{A}\left(\frac{1}{T}\sum_{t=-1}^{T-2}\expect{\gamma_t}\right) - \mathbf{b} }\geq f(\gamma^*)
\]
By Jensen's inequality, this implies,
\begin{equation}\label{eq:inter-lagrange-1}
\frac{1}{T}\sum_{t=-1}^{T-2}\expect{f\l(\gamma_t\r)} + \dotp{\lambda}{\mathbf{A}\l(\frac{1}{T}\sum_{t=-1}^{T-2}\expect{\gamma_t}\r) - \mathbf{b} }\geq f(\gamma^*)
\end{equation}
By Lemma \ref{lem:diff} we have
\begin{align*}
\expect{\overline{x}_T} -\frac{1}{T}\sum_{t=-1}^{T-2}\expect{\gamma_t} 
= \frac{\expect{\gamma_{T-1}}}{\eta T}
\end{align*}
By \eqref{eq:inter-lagrange-1}, this implies 
\begin{multline*}\label{eq:inter-lagrange-2}
\frac{1}{T}\sum_{t=-1}^{T-2}\expect{f\l(\gamma_t\r)} - f(\gamma^*) \geq - \dotp{\lambda}{\mathbf{A}\expect{\overline{x}_T} - \mathbf{b} } + \frac{1}{\eta T} \dotp{\lambda}{\mathbf{A}\expect{\gamma_{T-1}}  }\\
\geq - \dotp{\lambda}{\mathbf{A}\expect{\overline{x}_T} - \mathbf{b} }
- \frac{1}{\eta T}\|\mathbf A^T\lambda\|D
\end{multline*}
Substituting this bound into \eqref{eq:inter-last} we have
\begin{multline*}
- \dotp{\lambda}{\mathbf{A} \expect{\overline{x}_T}-\mathbf{b} }
- \frac{1}{\eta T}\|\mathbf A^T\lambda\|D \\
\leq
- \frac{\expect{f(\gamma_{T-1})}}{\eta T }+\frac{\expect{f(\gamma_{-1})}}{\eta T} 
- \frac{\expect{\|\mathbf Q(T)\|^2-\|\mathbf Q(0)\|^2}}{2V T} 
+ \frac{B^2}{2V} + \frac{LD^2}{2}\eta.
\end{multline*}
By Lemma \ref{lem:q-constraint}, this further implies
\begin{multline*}
-\frac{\|\lambda\| \expect{\|\mathbf{Q}(T)\|}}{T}-\frac{1}{\eta }\|\mathbf A^T\lambda\|D\\
\leq
- \frac{\expect{f(\gamma_{T-1})}}{\eta T }+\frac{\expect{f(\gamma_{-1})}}{\eta T } 
- \frac{\expect{\|\mathbf Q(T)\|^2-\|\mathbf Q(0)\|^2}}{2V T} 
+ \frac{B^2}{2V} + \frac{LD^2}{2}\eta.
\end{multline*}
Rearranging the terms and substituting $\mathbf{Q}(0) = 0$, $|f(\gamma)|\leq K$ gives
\[
\expect{\|\mathbf Q(T)\|^2} -2  V\|\lambda\|\cdot\expect{\|\mathbf Q(T)\|}-\frac{2\|\mathbf A^T\lambda\|DV}{\eta}- \frac{4KV}{\eta}-B^2T - LD^2\eta VT \leq 0.
\]
By Jensen's inequality,
\[
\expect{\|\mathbf Q(T)\|}^2 -2 V\|\lambda\|\cdot\expect{\|\mathbf Q(T)\|}-\frac{2\|\mathbf A^T\lambda\|DV}{\eta}- \frac{4KV}{\eta}-B^2T - LD^2\eta VT \leq 0.
\]
Solving the quadratic inequality with respect to $\expect{\|\mathbf Q(T)\|}$ gives
\begin{align*}
\expect{\|\mathbf Q(T)\|}\leq& V\|\lambda\| + \sqrt{V^2\|\lambda\|^2+\frac{2\|\mathbf A^T\lambda\|DV+ 4KV}{\eta}+B^2T + LD^2\eta T}\\ 
\leq&2 V\|\lambda\| + \sqrt{\frac{2\|\mathbf A^T\lambda\|DV+ 4KV}{\eta}+B^2T + LD^2\eta V T},
\end{align*}
finishing the proof.
\end{proof}

\begin{theorem}\label{col:convex-1}
Let $V=\sqrt{T}$ and $\eta=\frac{1}{\sqrt{T}}$, then, we have objective and constraint violation bounds as follows,
\begin{align*}
&\expect{f(\overline{x}_T)}\leq f(\gamma^*) + \l( 2K + MD + \frac{B^2}{2} + \frac{LD^2}{2}\r)\frac{1}{T^{1/2}},\\
&\dotp{\mathbf{a}_i}{\expect{\overline{x}_T}} - b_i \leq \frac{2\|\lambda\|}{\sqrt{T}}+ \sqrt{\frac{2\|\mathbf A^T\lambda\|D+ 4K+B^2 + LD^2}{T}},~\forall i\in\{1,2,\cdots,N\}.
\end{align*}
and by Jensen's inequality the objective function bound also holds for $f(\expect{\overline{x}_{T}})$. 
Thus, to achieve an $\varepsilon$-near optimality, the convergence time is $\mathcal{O}(1/\varepsilon^2)$.
\end{theorem}
\begin{proof}
The objective bound follows directly from Lemma \ref{thm:convex-2}. The constraint violation bound follows from Lemma \ref{thm:convex-2}  and Lemma \ref{lem:q-constraint}.
\end{proof}

\section{Performance bounds for non-convex cost functions}

In this section, we focus on the case when the cost function $f(\cdot)$ is smooth but non-convex. 
The same algorithm shall be used, which generates vectors $\{x_0, x_1, \ldots, x_{T-1}\}$ and 
$\{\gamma_{-1},\gamma_0, \gamma_1, \ldots, \gamma_{T-1}\}$ over the course of $T$ slots.   However, 
instead of having $\overline{x}_T$ as the output, we focus on a \textbf{randomized solution $\gamma_\alpha$}, 
where $\alpha$ is a
uniform random variable taking values in $\{-1,0,1,\cdots,T-2\}$ and independent of any other event in the system. 
That is, the output $\gamma_{\alpha}$ is chosen uniformly over the set 
$\{\gamma_{-1}, \gamma_0, \ldots, \gamma_{T-2}\}$.

\subsection{Preliminaries}
Since the function $f(\cdot)$ is non-convex, finding a global solution is difficult. Thus, the criterion used for convergence analysis is important in non-convex optimization. In an unconstrained optimization problem, the magnitude of the gradient, i.e. $\|\nabla f\|$ is usually adopted to measure the convergence of the algorithm since $\nabla f(x) = 0$ implies $x$ is a stationary point. However, this cannot be used in constrained problems. Instead, the criterion for our problem (\ref{prob-1}-\ref{prob-3}) is the following:\footnote{This quantity is sometimes referred to as the ``Frank-Wolfe gap'' because it plays an important role in the analysis of Frank-Wolfe type algorithms.}
\begin{equation}\label{eq:FW-gap}
\mathcal{G}(\gamma) := \sup_{\{v\in\overline{\Gamma}^* : \mathbf{A}v\leq\mathbf{b}\}}\dotp{\nabla f(\gamma)}{\gamma - v}. 
\end{equation}
When the function $f(\cdot)$ is convex, this quantity upper bounds the suboptimality regarding the problem (\ref{prob-1}-\ref{prob-3}), which is an obvious outcome of \eqref{eq:convexity}. In the case of non-convex $f(\cdot)$, any point $\gamma\in\mathbb{R}^d$ which yields $\mathcal{G}(\gamma) = 0$ is a stationary point of (\ref{prob-1}-\ref{prob-3}). In particular, if the function $f$ is strictly increasing on each coordinate, then, any point $\gamma$ satisfying $\gamma\in\overline{\Gamma}^*,~\mathbf{A}\gamma\leq\mathbf{b}$ and $\mathcal{G}(\gamma) = 0$ must be on the boundary of the feasible set $\l\{v\in\mb R^d:~v\in\overline{\Gamma}^*,~\mathbf{A}v\leq\mathbf{b}\r\}$.
See, for example \cite{lacoste2016convergence} and \cite{reddi2016stochastic}, for more discussions and applications of this criterion.

Recall in the convex case, we proved performance bounds regarding $\expect{\overline{x}_T}\in\overline{\Gamma}^*$, so in the non-convex case, ideally, one would like to show similar bounds for $\expect{\gamma_\alpha}\in\overline{\Gamma}^*$.
However, 
proving such performance bounds turns out to be difficult for general non-convex stochastic problems. 
In particular, Jensen's inequality does not hold anymore, which prohibits us from passing expectations into the function.\footnote{Specifically, for non-convex $f$, performance bounds on $\expect{f(\gamma_\alpha)}$ do not imply  guarantees on $f(\expect{\gamma_\alpha})$. In contrast, if $f$ is convex then $\expect{f(\gamma_\alpha)}\geq f(\expect{\gamma_\alpha})$.} To obtain a meaningful performance guarantee for non-convex problems, we instead try to characterize the violation of $\gamma_\alpha$ over $\overline{\Gamma}^*$, which is measured through the following quantity:
\begin{definition}
For any point $\gamma\in\mathbb{R}^d$ and any closed convex set $\mathcal{S}\subset\mathbb{R}^d$, the distance of $\gamma$ to the set $\mathcal{S}$ is defined as
\begin{equation}\label{eq:distance}
\text{dist}\l(\gamma, \mathcal{S}\r):= \min_{v\in\mathcal{S}}\|\gamma-v\|.
\end{equation}
\end{definition}


Now, we are ready to state our definition of $\varepsilon$-near local optimal solution for non-convex $f$ functions.

\begin{definition}
Given an accuracy level $\varepsilon>0$, a random vector $\gamma\in\mathbb{R}^d$ is said to be an $\varepsilon$-near local optimal solution regarding 
(\ref{prob-1}-\ref{prob-3}) if it satisfies
\begin{align*}
&\expect{\mathcal{G}(\gamma)}\leq \varepsilon,\\
&\expect{\text{dist}\l(\gamma, \overline{\Gamma}^*\r)^2}\leq \varepsilon,\\
&\dotp{\mathbf{a}_i}{\expect{\gamma}}-b_i\leq\varepsilon,~\forall i\in\{1,2,\cdots,N\}.
\end{align*}
\end{definition}

\subsection{Convergence time bound}
Recall that we use $\gamma_\alpha$ as the output, where $\alpha$ is uniform random variable taking values in $\{-1,0,1,\cdots,T-2\}$.
We first bound the mean square distance of $\gamma_\alpha$ to the set $\overline{\Gamma}^*$, i.e.
$\expect{\text{dist}\l(\gamma_\alpha, \overline{\Gamma}^*\r)^2}$. To do so, we will explicitly construct a random feasible vector $\widetilde{\gamma}_\alpha\in \overline{\Gamma}^*$ and try to bound the distance between $\gamma_\alpha$ and $\widetilde{\gamma}_\alpha$, which is entailed in the following lemma.

\begin{lemma}\label{lem:path-average}
For all $T>0$, the output $\gamma_\alpha$ satisfies
\[
\expect{\text{dist}\l(\gamma_\alpha, \overline{\Gamma}^*\r)^2}\leq \eta D^2,
\]
\end{lemma}
\begin{proof}[Proof of Lemma \ref{lem:path-average}]
Recall that $\mathcal{H}_t$ is the system history up until time slot $t$. Let $\widetilde{x}_t = \expect{x_t | \mathcal{H}_t}$ be the conditional expectation given the system history. Note that $\widetilde{x}_t \in\overline{\Gamma}^*$. To see this, by \eqref{sub-problem}, the decision $x_t$ is completely determined by $\gamma_{t-1}$, $\mathbf{Q}(t)$ and $\mathcal{X}_{S[t]}$. Since $S[t]$ is i.i.d., it is independent of $\gamma_{t-1}$, $\mathbf{Q}(t)$, and thus the decision at time $t$ is generated by a randomized stationary algorithm which always uses the same fixed $\gamma=\gamma_{t-1}$, $\mathbf{Q}=\mathbf{Q}(t)$ values and at each time slot $t'\in\{0,1,2,\cdots\}$ chooses $x_{t'}\in\mathcal{X}_{S[t']}$ according to
\[
x_{t'}:=\argmin_{x\in\mathcal{X}_{S[t']}}    V\dotp{\nabla f(\gamma)}{x} + \sum_{i=1}^NQ_i\dotp{\mathbf{a}_i}{x}.
\]

Let $\widetilde{\gamma}_{-1} = \gamma_{-1} = 0\in\overline{\Gamma}^*$ by the assumption that $\overline{\Gamma}^*$ contains the origin.
For any fixed $t\in\{0,1,2,\cdots,T-1\}$, we iteratively define the following path averaging sequence $\widetilde{\gamma}_t$ as follows:
\[
\widetilde{\gamma}_t = (1-\eta)\widetilde{\gamma}_{t-1} + \eta \widetilde{x}_t.
\]
Note that the sequence $\widetilde{\gamma}_t\in\overline{\Gamma}^*,~\forall t=0,1,2,\cdots, T$ due to the fact that $\overline{\Gamma}^*$ is convex. We first show that $\gamma_t$ and $\widetilde\gamma_t$ are close to each other. Specifically, we use induction to show
\[
\expect{\|\gamma_t-\widetilde\gamma_t\|^2}\leq \eta D^2,~\forall t\in\{0, 1,2,\cdots\}.
\]
For $t = 0$, we have
\[
\expect{\|\gamma_0 - \widetilde\gamma_0\|^2} = \eta^2\expect{\|x_0 - \widetilde{x}_0\|^2}\leq \eta^2 D^2 \leq \eta D^2.
\]
For any $t\geq1$, we assume the claim holds for any slot before $t$,
\begin{align*}
&\expect{\|\gamma_t-\widetilde\gamma_t\|^2}\\
=&\expect{\|(1-\eta)(\gamma_{t-1}-\widetilde\gamma_{t-1}) + \eta(x_t-\widetilde x_t)\|^2} \\
=& \eta^2\expect{\|x_t-\widetilde x_t\|^2} + 2\eta(1-\eta)\expect{\dotp{\gamma_{t-1}-\widetilde\gamma_{t-1}}{x_t-\widetilde x_t}}
+(1-\eta)^2\expect{\| \gamma_{t-1}-\widetilde\gamma_{t-1} \|^2}\\
\overset{(a)}{=}&\eta^2\expect{\|x_t-\widetilde x_t\|^2}+(1-\eta)^2\expect{\| \gamma_{t-1}-\widetilde\gamma_{t-1} \|^2}\\
\leq&\eta^2D^2 + (1-\eta)^2\expect{\| \gamma_{t-1}-\widetilde\gamma_{t-1} \|^2}\\
\overset{(b)}{\leq}&\eta^2D^2 + (1-\eta)^2\eta D^2\leq \eta^2D^2 + (1-\eta)\eta D^2 = \eta D^2,
\end{align*}
where (a) follows from the fact that
\begin{align*}
\expect{\dotp{\gamma_{t-1}-\widetilde\gamma_{t-1}}{x_t-\widetilde x_t}}
= &\expect{\expect{\dotp{\gamma_{t-1}-\widetilde\gamma_{t-1}}{x_t-\widetilde x_t}|\mathcal{H}_t}}\\
=&\expect{\dotp{\gamma_{t-1}-\widetilde\gamma_{t-1}}{\widetilde x_t-\widetilde x_t}} = 0,
\end{align*}
while (b) follows from the induction hypothesis. 

Since $\alpha$ is a uniform random variable taking values in $\{-1, 0,1,\cdots,T-1\}$, it follows
\[
\expect{\|\gamma_\alpha-\widetilde\gamma_\alpha\|^2} = \frac{1}{T}\sum_{t=-1}^{T-2}\expect{\|\gamma_t-\widetilde\gamma_t\|^2}\leq \eta D^2,
\]
which implies the claim by the fact that $\widetilde\gamma_{\alpha}\in\overline{\Gamma}^*$ and $\expect{\text{dist}\l(\gamma_\alpha, \overline{\Gamma}^*\r)^2}\leq \expect{\|\gamma_\alpha-\widetilde\gamma_\alpha\|^2}$.
\end{proof}

\begin{lemma}\label{thm:main}
For any $T>1$, we have
\begin{align*}
&\expect{\mathcal{G}(\gamma_\alpha)} \leq \frac{2K}{\eta T} + \frac{B^2}{ V} + \frac{\eta LD^2}{2},\\
&\frac{\expect{\|\mathbf{Q}(T)\|^2}}{T^2}\leq 2MD \frac{ V}{T} + 4K\frac{V}{\eta T^2} + \frac{B^2}{T} + LD^2\frac{\eta V}{T}.
\end{align*}
\end{lemma}

\begin{proof}[Proof of Theorem \ref{thm:main}]
Recall that $\Delta(t) = \l(\|\mathbf{Q}(t+1)\|^2-\|\mathbf{Q}(t)\|^2\r)/2,~\forall t\in \{0,1,2, \ldots\}$. By the smoothness property of function $f(\cdot)$, we have for any $t\geq1$,
\begin{align}
Vf(\gamma_t)+ \eta\Delta(t) \leq& Vf(\gamma_{t-1})+V\dotp{\nabla f(\gamma_{t-1})}{\gamma_t-\gamma_{t-1}}+\frac{VL}{2}\|\gamma_t-\gamma_{t-1}\|^2\nonumber+ \eta\Delta(t)\\
=& Vf(\gamma_{t-1}) + V\eta\dotp{\nabla f(\gamma_{t-1})}{x_t-\gamma_{t-1}} +\frac{VL\eta^2}{2}\|x_t-\gamma_{t-1}\|^2+ \eta\Delta(t)\nonumber\\
\leq&V f(\gamma_{t-1}) + V\eta\dotp{\nabla f(\gamma_{t-1})}{x_t-\gamma_{t-1}} +\frac{VL\eta^2D^2}{2}+ \eta\Delta(t),\label{eq-1}
\end{align}
where the equality follows from the updating rule that $\gamma_t = (1-\eta)\gamma_{t-1}+\eta x_t$. On the other hand, we have the following bound on the drift $\Delta(t)$,
\begin{multline*}
\Delta(t) = \frac12\l( \|\mathbf Q(t+1)\|^2 - \|\mathbf Q(t)\|^2 \r)
\leq\frac12\l( \sum_{i=1}^N\l(Q_i(t) + \dotp{\mathbf{a}_i}{x_t} - b_i \r)^2 - \sum_{i=1}^NQ_i(t)^2\r)\\
= \frac12\sum_{i=1}^N\l( \dotp{\mathbf{a}_i}{x_t} - b_i \r)^2 + \sum_{i=1}^NQ_i(t)\l( \dotp{\mathbf{a}_i}{x_t} - b_i \r)
\leq  \frac{B^2}{2} + \dotp{\mathbf Q(t)}{\mathbf{A}x_t-\mathbf{b}}.
\end{multline*}

Substituting this bound into \eqref{eq-1} gives
\begin{multline}\label{eq:unconditional}
Vf(\gamma_t)+ \eta\Delta(t) \leq
V f(\gamma_{t-1}) +
V\eta\dotp{\nabla f(\gamma_{t-1})}{x_t-\gamma_{t-1}} +\eta \dotp{\mathbf Q(t)}{\mathbf{A}x_t-\mathbf{b}}\\
+\frac{VL\eta^2D^2}{2} + \frac{\eta B^2}{2}
\end{multline}

Taking conditional expectation from both sides of  \eqref{eq:unconditional} conditioned on $\mathcal{H}_t$, we have
\begin{multline*}
V\expect{f(\gamma_t)|\mathcal{H}_t}+ \eta\expect{\Delta(t)|\mathcal{H}_t} \\
\leq V f(\gamma_{t-1}) +
V\eta\expect{\dotp{\nabla f(\gamma_{t-1})}{x_t-\gamma_{t-1}} | \mathcal{H}_t} +\eta \expect{\dotp{\mathbf Q(t)}{\mathbf{A}x_t-\mathbf{b}}|\mathcal{H}_t}
+\frac{VL\eta^2D^2}{2} + \frac{\eta B^2}{2}.
\end{multline*}
Substituting Lemma \ref{lem:key-equation} into the right hand side gives
\begin{align*}
V\expect{f(\gamma_t)|\mathcal{H}_t}+ \eta\expect{\Delta(t)|\mathcal{H}_t} 
\leq V f(\gamma_{t-1}) +
V\eta\dotp{\nabla f(\gamma_{t-1})}{ \gamma^*-\gamma_{t-1}}
+\frac{VL\eta^2D^2}{2} + \frac{\eta B^2}{2},
\end{align*}
where $\gamma^*$ is any point satisfying  $\gamma\in\overline{\Gamma}^*$ and $\mathbf{A}\gamma\leq \mathbf{b}$. Rearranging the terms gives
\[
\dotp{\nabla f(\gamma_{t-1})}{ \gamma_{t-1}-\gamma^*}\leq \frac{ f(\gamma_{t-1}) - \expect{f(\gamma_t)|\mathcal{H}_t} }{\eta} - 
\frac{\expect{\Delta(t)|\mathcal{H}_t} }{V} +\frac{\eta LD^2}{2} + \frac{B^2}{2 V}
\]
Since this holds for any $\gamma^*$ satisfying $\gamma\in\overline{\Gamma}^*$ and $\mathbf{A}\gamma\leq \mathbf{b}$, we can take the supremum on the left-hand-side over all feasible points and the inequality still holds. This implies the gap function on the point $\gamma_{t-1}$ satisfies,
\begin{multline}\label{eq:dpp-bound-1}
\mathcal{G}(\gamma_{t-1}) = \sup_{\gamma\in\overline{\Gamma}^*,~\mathbf{A}\gamma\leq\mathbf{b}}\dotp{\nabla f(\gamma_{t-1})}{\gamma_{t-1} - \gamma}
\leq  \frac{ f(\gamma_{t-1}) - \expect{f(\gamma_t)|\mathcal{H}_t} }{\eta} - 
\frac{\expect{\Delta(t)|\mathcal{H}_t} }{ V} +\frac{\eta LD^2}{2} + \frac{B^2}{2 V}.
\end{multline}
Now, take the full expectation on both sides and sum over $t \in \{1, \ldots, T\}$  to obtain
\begin{align*}
\frac1T\sum_{t=0}^{T-1}\expect{\mathcal{G}(\gamma_{t})} \leq &
\frac{1}{T\eta}(  \expect{f(\gamma_{0})} - \expect{f(\gamma_T)} ) - \frac{\expect{\|\mathbf{Q}(T+1)\|^2 - \|\mathbf{Q}(1)\|^2}}{2 VT}+\frac{D^2L\eta}{2} + \frac{B^2}{2V}\\
\leq& \frac{2K}{\eta T} + \frac{B^2}{ V} + \frac{\eta LD^2}{2},
\end{align*}
where the second inequality follows from $\|\mathbf Q(1)\| = \|\mathbf{A}x_0 - \mathbf{b}\|\leq B$ and $T\geq 1$. 
Finally, since $\alpha$ is a uniform random variable on $\{0,1,2,\cdots,T-1\}$, independent of any other random events in the system, it follows 
$\expect{\mathcal{G}(\gamma_\alpha)}$ satisfies the same bound.

To get the $\mathbf{Q}(T)$ bound, we rearrange the terms in \eqref{eq:dpp-bound-1} and take full expectations, which yields
\begin{align*}
\expect{\|\mathbf{Q}(t+1)\|^2 - \|\mathbf{Q}(t)\|^2} \leq& 2 V \expect{\mathcal{G}(\gamma_{t-1})} + \frac{2V}{\eta}\left(\expect{f(\gamma_{t-1})} - \expect{f(\gamma_t)}\right) 
+ B^2 + \eta VLD^2\\
\leq& 2 VMD +\frac{2V}{\eta}\left(\expect{f(\gamma_{t-1})} - \expect{f(\gamma_t)}\right) + B^2 + \eta VLD^2.
\end{align*}
Summing both sides over $t \in \{0, \ldots, T-1\}$  and using $\mathbf{Q}(0)=0$ gives
\begin{align*}
\expect{\|\mathbf{Q}(T)\|^2}\leq &
2 VTMD +\frac{2V}{\eta}\left(\expect{f(\gamma_{-1})} - \expect{f(\gamma_{T-1})}\right) + B^2T + \eta VTLD^2\\
\leq& 2 VTMD +\frac{4VK}{\eta} + B^2T + \eta VTLD^2.
\end{align*}
Dividing both sides by  $T^2$  gives the $\mathbf{Q}(T)$ bound.
\end{proof}

\begin{theorem}\label{col:non-convex}
Let $V=T^{1/3}$ and $\eta=\frac{1}{T^{2/3}}$, then, we have objective and constraint violation bounds as follows,
\begin{align*}
&\expect{\mathcal{G}(\gamma_\alpha)} \leq \l( 2K + B^2 + \frac{LD^2}{2 T^{1/3}}\r)\frac{1}{T^{1/3}},\\
&\dotp{\mathbf{a}_i}{\expect{\gamma_\alpha}} - b_i \leq \left(\sqrt{2MD+\frac{4K+B^2}{T^{1/3}} + \frac{LD^2}{T^{2/3}} }+ \|\mathbf{a}_i\|D\right)\frac{1}{T^{1/3}},~\forall i\in\{1,2,\cdots,N\}.\\
&\expect{\text{dist}\l(\gamma_\alpha, \overline{\Gamma}^*\r)^2}\leq \frac{D^2}{T^{2/3}},
\end{align*}
Thus, to achieve an $\varepsilon$-near local optimality, the convergence time is $\mathcal{O}(1/\varepsilon^3)$.
\end{theorem}

Before giving the proof, it is useful to present the following  straight-forward corollary that applies the above theorem 
to the special case of deterministic problems where $\mathcal{X}_{S[t]} = \overline{\Gamma}^*,~\forall t$, i.e. $\mathcal{X}_{S[t]}$ is fixed. In this scenario, the only randomness comes from the algorithm of choosing the final solution, and it is obvious that $\gamma_t\in\overline{\Gamma}^*,~\forall t=-1,0,1,\cdots,T$. As a consequence, if we choose any $\gamma_t$ as a solution, then, $\text{dist}\l(\gamma_t, \overline{\Gamma}^*\r) =0$ and we do not have to worry about the infeasibility with respect to $\overline{\Gamma}^*$ as we do for the stochastic case.

\begin{corollary}[Deterministic programs]\label{cor:determine}
Suppose $\mathcal{X}_{S[t]} = \overline{\Gamma}^*,~\forall t$, i.e. $\mathcal{X}_{S[t]}$ is fixed, then, choose $V=T^{1/3}$, $\eta=\frac{1}{T^{2/3}}$,
\begin{align*}
&\expect{\mathcal{G}(\gamma_{\alpha})} \leq \l( 2K + B^2 + \frac{LD^2}{2T^{1/3}}\r)\frac{1}{T^{1/3}},\\
&\dotp{\mathbf{a}_i}{\expect{\gamma_{\alpha}}} - b_i \leq \left(\sqrt{2MD+\frac{4K + B^2}{T^{1/3}} + \frac{LD^2}{T^{2/3}} }+ \|\mathbf{a}_i\|D\right)\frac{1}{T^{1/3}},~\forall i\in\{1,2,\cdots,N\}.
\end{align*}
\end{corollary}

\begin{proof}[Proof of Theorem \ref{col:non-convex}]
The $\expect{\mathcal{G}(\gamma_\alpha)}$ and $\expect{\text{dist}\l(\gamma_\alpha, \overline{\Gamma}^*\r)^2}$ bounds 
follow directly from Lemma \ref{thm:main} and Lemma \ref{lem:path-average}. To get the constraint violation bound, we take the square root from both sides of the 
$\|\mathbf{Q}(T)\|^2$ bound, which gives
\[
\frac{\expect{\|\mathbf{Q}(T)\|}}{T}\leq \frac{\sqrt{2MD+\frac{4K + B^2}{T^{1/3}}+ \frac{LD^2}{T^{2/3}} }}{T^{1/3}}
\]
By Lemma \ref{lem:q-constraint}, we have


\begin{equation} \label{eq:quick-equation} 
\dotp{\mathbf{a}_i}{\expect{\overline{x}_T}} - b_i \leq
 \frac{\sqrt{2MD+\frac{4K + B^2}{T^{1/3}}+ \frac{LD^2}{T^{2/3}} }}{T^{1/3}},~\forall i=1,2,\cdots,N.
\end{equation} 
By Lemma \ref{lem:diff},
\begin{multline*}
\l|\dotp{\mathbf{a}_i}{\expect{\frac1T\sum_{t= 0}^{T-1}{\gamma}_{t-1}}} - \dotp{\mathbf{a}_i}{\expect{\overline{x}_T}}\r|
\leq \|\mathbf{a}_i\|  \expect{\l\|\frac1T\sum_{t= 0}^{T-1}{\gamma}_{t-1} - \overline x_T\r\|}
\leq \frac{\|\mathbf{a}_i\|D}{\eta T} = \frac{\|\mathbf{a}_i\|D}{ T^{1/3}}
\end{multline*}
which implies 
$$
\dotp{\mathbf{a}_i}{\expect{\gamma_\alpha}} - b_i
 = \dotp{\mathbf{a}_i}{\expect{\frac1T\sum_{t= 0}^{T-1}{\gamma}_{t-1}}} - b_i
 \leq  \dotp{\mathbf{a}_i}{\expect{\overline{x}_T}} - b_i + \frac{\|\mathbf{a}_i\|D}{ T^{1/3}}
$$
which, when combined with \eqref{eq:quick-equation}, 
finishes the proof.
\end{proof}

\subsection{An improved convergence time bound via Slater's condition}


In this section, we show that a better convergence time bound is achievable when the Slater's condition hold. Specifically, we have the following assumption:
\begin{assumption}[Slater's condition]\label{ass:slater}
There exists a randomized stationary policy of choosing $\widetilde{ x}_t$ from $\mathcal{X}_{S[t]}$ at each time slot $t$, such that
\[
\dotp{\mf a_i}{\expect{\widetilde{x}_t}} - b_i \leq -\varepsilon,
\] 
for some fixed constant $\varepsilon>0$ and any $i\in\l\{ 1,2,\cdots,N \r\}$.
\end{assumption}
We have the following convergence time bound:
\begin{theorem}\label{thm:slater}
Suppose Assumption \ref{ass:slater} holds, then, for any $T\geq1$, choosing $\eta = 1/\sqrt{T}$ and $V=\sqrt T$,
\begin{align*}
&\expect{\mathcal{G}(\gamma_\alpha)}\leq (2K + B^2 + LD^2)\frac{1}{\sqrt{T}},\\
&\dotp{\mf a_i}{\expect{\gamma_\alpha}} - b_i\leq \frac{C_i}{\sqrt T},~~\forall i\\
&\expect{\text{dist}\l(\gamma_\alpha, \overline{\Gamma}^*\r)^2}\leq \frac{D^2}{\sqrt{T}}.
\end{align*}
where 
\[
C_i =  \frac{B^2+LD^2+2MD}{\varepsilon} + B+4K+\varepsilon + \frac{8B^2}{\varepsilon}
 \cdot \ln\l( 1+ \frac{32B^2}{\varepsilon^2}e \r) + \|\mf a_i\|D .
\]
Thus, to achieve an $\varepsilon$-near local optimality, the convergence time is $\mathcal{O}(1/\varepsilon^2)$.
\end{theorem}
Thus, under Slater's condition, we can choose $\eta, V$ to get a better trade-off on the objective suboptimality and constraint violation bounds compared to Theorem \ref{col:non-convex} (improve the rate from $1/\varepsilon^{3}$ to $1/\varepsilon^2$). Furthermore, as a straight-forward corollary, we readily obtain an improvement in the deterministic scenario compared to Corollary \ref{cor:determine}.
\begin{corollary}[Deterministic programs]\label{cor:determine-2}
Suppose $\mathcal{X}_{S[t]} = \overline{\Gamma}^*,~\forall t$, i.e. $\mathcal{X}_{S[t]}$ is fixed, then, choose $V=\sqrt T$, $\eta=\frac{1}{\sqrt T}$,
\begin{align*}
&\expect{\mathcal{G}(\gamma_{\alpha})} \leq \l( 2K + B^2 + \frac{LD^2}{2}\r)\frac{1}{\sqrt T},\\
&\dotp{\mathbf{a}_i}{\expect{\gamma_{\alpha}}} - b_i \leq \frac{C_i}{\sqrt T},~\forall i\in\{1,2,\cdots,N\},
\end{align*}
where 
\[
C_i =  \frac{B^2+LD^2+2MD}{\varepsilon} + B+4K+\varepsilon + \frac{8B^2}{\varepsilon}
 \cdot \ln\l( 1+ \frac{32B^2}{\varepsilon^2}e \r) + \|\mf a_i\|D.
\]
\end{corollary}
It is also worth noting that such a rate matches the $1/\sqrt{T}$ Frank-Wolfe gap convergence rate established in the deterministic scenario without linear constraints (e.g. \cite{lacoste2016convergence}).

The key argument proving Theorem  \ref{thm:slater} is to derive a tight bound on the virtual queue term $\|\mf Q(t)\|$ via Slater's condition. Intuitively, having Slater's condition ensures the property that the virtual queue process $\|\mf Q(t)\|$ will strictly decrease whenever it gets large enough. Such a property is often referred to as the ``drift condition''. 
Any process satisfying such a drift condition enjoys a tight upper bound, as is shown in the following lemma:
\begin{lemma}[Lemma 5 of \cite{yu2017online}]\label{lem:drift}
Let $\l\{Z(t),t\in\mb N  \r\}$ be a discrete time stochastic process adapted to a filtration $\l\{\mathcal F_t,~t\in\mb N\r\}$ such that $Z(0) = 0$. If there exist an integer $t_0>0$ and real constants $\lambda,~\delta_{\max}>0$ and $0 < \xi <\delta_{\max}$ such that
\begin{align*}
&|Z(t+1) - Z(t)|\leq \delta_{\max},\\
&\expect{Z(t+t_0) - Z(t) | \mathcal{F}_t}\leq
\begin{cases}
t_0\delta_{\max},~~&\text{if}~Z(t)\leq \lambda,\\
-t_0\xi,~~&\text{if}~Z(t)>\lambda.
\end{cases}
\end{align*}
Then, we have for any $t\in\mb N$,
\[
\expect{Z(t)} \leq \lambda + \frac{4\delta_{\max}^2}{\xi}t_0\cdot\ln\l( 1+\frac{8\delta_{\max}^2}{\xi^2}e^{\xi/4\delta_{\max}} \r)
\]
\end{lemma}
Note that the special case $t_0= 1$ of this lemma appears in the earlier work \cite{wei2015probabilistic}, which is used to prove tight convergence time for the ``drift-plus-penalty'' algorithm.
Our goal here is to verify that the process $\|\mf Q(t)\|$ satisfies the assumption specified in Lemma \ref{lem:drift}, which is done in the following lemma. The proof is delayed to the appendix.
\begin{lemma}\label{lem:drift-2}
The process $\l\{\|\mf Q(t)\|\r\}_{t=0}^{\infty}$ satisfies the assumption in Lemma \ref{lem:drift} for any $t_0\geq1$, $\delta_{\max} = B$, $\xi = \varepsilon/2$ and 
$\lambda = \frac{(VL\eta D^2 + B^2+2VMD)t_0+4VK/\eta + \varepsilon Bt_0^2+\varepsilon^2t_0^2}{\varepsilon t_0}$.
\end{lemma}

\begin{proof}[Proof of Theorem \ref{thm:slater}]
By the drift lemma (Lemma \ref{lem:drift}) and Lemma \ref{lem:drift-2}, we have 
\[
\expect{\|\mf Q(t)\|}\leq 
 \frac{(VL\eta D^2 + B^2+2VMD)t_0+4VK/\eta + \varepsilon Bt_0^2+\varepsilon^2t_0^2}{\varepsilon t_0} + \frac{8B^2}{\varepsilon}t_0
 \cdot \ln\l( 1+ \frac{32B^2}{\varepsilon^2}e \r).
\]
Take $\eta = 1/\sqrt{T},~V=\sqrt{T}$ and $t_0=\sqrt{T}$ gives
\[
\expect{\|\mf Q(t)\|}\leq\l(\frac{B^2+LD^2+2MD}{\varepsilon} + B+4K+\varepsilon + \frac{8B^2}{\varepsilon}
 \cdot \ln\l( 1+ \frac{32B^2}{\varepsilon^2}e \r)\r)\sqrt{T}.
\]
By Lemma \ref{lem:q-constraint}, we have
\begin{equation}\label{constraint-vio}
\dotp{\mf a_i}{\expect{\overline{x}_T}} - b_i\leq 
\l(\frac{B^2+LD^2+2MD}{\varepsilon} + B+4K+\varepsilon + \frac{8B^2}{\varepsilon}
 \cdot \ln\l( 1+ \frac{32B^2}{\varepsilon^2}e \r)\r)\frac{1}{\sqrt{T}}.
\end{equation}
By Lemma \ref{lem:diff},
\begin{multline*}
\l|\dotp{\mathbf{a}_i}{\expect{\frac1T\sum_{t= 0}^{T-1}{\gamma}_{t-1}}} - \dotp{\mathbf{a}_i}{\expect{\overline{x}_T}}\r|
\leq \|\mathbf{a}_i\|  \expect{\l\|\frac1T\sum_{t= 0}^{T-1}{\gamma}_{t-1} - \overline x_T\r\|}
\leq \frac{\|\mathbf{a}_i\|D}{\eta T} = \frac{\|\mathbf{a}_i\|D}{ T^{1/3}}
\end{multline*}
which implies 
$$
\dotp{\mathbf{a}_i}{\expect{\gamma_\alpha}} - b_i
 = \dotp{\mathbf{a}_i}{\expect{\frac1T\sum_{t= 0}^{T-1}{\gamma}_{t-1}}} - b_i
 \leq  \dotp{\mathbf{a}_i}{\expect{\overline{x}_T}} - b_i + \frac{\|\mathbf{a}_i\|D}{ T^{1/3}}
$$
which, by substituting in \eqref{constraint-vio}, implies the constraint violation bound. The objective suboptimality bound and distance to feasibility bound follow directly from Lemma \ref{thm:main} and \ref{lem:path-average} with $\eta = 1/\sqrt{T},~V=\sqrt{T}$.
\end{proof}

\subsection{Time average optimization in non-convex opportunistic scheduling}
In this section, we consider the application of our algorithm in opportunistic scheduling and show how one can design a policy to achieve $\gamma_\alpha$ computed by our proposed algorithm. 

We consider a wireless system with $d$ users that transmit over their own links. The wireless channels can change over time and this affects the set of transmission rates available
for scheduling. Specifically, the random sequence $\l\{ S[t] \r\}_{t=0}^\infty$ can be a process of independent and identically distributed (i.i.d.) channel state vectors that take values in some set $\mathcal{S}\subseteq \mb R^d$. The decision variable $x_t$ is the transmission  rate vector chosen from $\Gamma_{S[t]}$, which is the set of available rates determined by the observation $S[t]$. 

We are interested in how well the system performs over   $T$ slots, i.e. whether or not the time average $\overline{x}_T = \frac1T\sum_{t=0}^{T-1}x_t$ minimizes $f(\cdot)$ while satisfying the constraints:
\begin{align*}
\min~~&\limsup_{T\rightarrow\infty}f(\expect{\overline{x}_T} )    \\
s.t.~~
&\limsup_{T\rightarrow\infty}\dotp{\mathbf{a}_i}{\expect{\overline{x}_T}}\leq b_i,~\forall i\in\{1,2,\cdots,N\} \\
&x_t\in\mathcal{X}_{S[t]},~\forall t\in\{0,1,2,\cdots\}.
\end{align*}

In the previous section, we show our algorithm produces an $\varepsilon$ near local optimal solution vector $\gamma_\alpha\in\overline{\Gamma}^*$ to (\ref{prob-1}-\ref{prob-3}) via a randomized selection procedure. Here, we show one can make use of the vector $\gamma_\alpha$ obtained by the proposed algorithm and rerun the system with a new sequence of transmission vectors $\l\{ x_t \r\}_{t=0}^{\infty}$ so that the time average $\overline{x}_T$ approximates the solution to the aforementioned time average problem.

Our idea is to run the system for $T$ slots, obtain $\gamma_\alpha$, and then run the system for another $T$ slots by choosing $x_t\in\mathcal{X}_{S[t]},~t=0,1,2,\cdots$ to solve the following problem:
\begin{align*}
\min ~~&\limsup_{T\rightarrow\infty}\|\mathbb{E}_{\alpha}[\overline{x}_T] - \gamma_\alpha\|^2\\
s.t.~~&x_t\in\mathcal{X}_{S[t]},~\forall t\in\{0,1,2,\cdots\},
\end{align*}
where $\mathbb E_\alpha[\cdot]$ denotes the the conditional expectation conditioned on the randomness in the first $T$ time slots generating $\gamma_\alpha$.
Interestingly, this is also a stochastic optimization problem with a smooth convex objective (without linear constraints), on which we can apply a special case of the primal-dual Frank-Wolfe algorithm with time varying weights $\eta_t$ and without virtual queues as follows: Let $\gamma_{-1} = 0$, $\eta_t = 1/(t+1)$ and at each time slot $t\in\{0,1,2,\cdots\}$
\begin{enumerate}
\item Choose $x_t$ as follows
\[x_t := \argmin_{x\in\mathcal{X}_{S[t]}}\dotp{\gamma_{t-1} - \gamma_{\alpha}}{x}\]
\item Update $\gamma_t$:
\[\gamma_t = (1-\eta_t)\gamma_{t-1} + \eta_t x_t.\]
\end{enumerate}
The following theorem, which is proved in \cite{neely2017frank}, gives the performance of this algorithm:
\begin{lemma}[Theorem 1 of \cite{neely2017frank}]\label{lem:improved-FW}
For any $T\geq1$, we have
\[
\l\|\mathbb{E}_{\alpha}[\overline{x}_T] - \gamma_\alpha\r\|^2 - \|\gamma^*-\gamma_\alpha\|^2\leq  \frac{D^2(1+\ln(T))}{T},
\]
where $\gamma^* := \argmin_{\gamma\in\overline{\Gamma}^*}\|\gamma - \gamma_\alpha\|$.
\end{lemma}
Note that by Theorem \ref{col:non-convex}, we have 
$$\expect{\|\gamma^*-\gamma_\alpha\|^2} = \expect{\text{dist}\l(\gamma_\alpha, \overline{\Gamma}^*\r)^2}\leq \frac{D^2}{T^{2/3}}.$$
Thus, Lemma \ref{lem:improved-FW} together with the above bound implies
\begin{equation}\label{eq:diff-bound}
\expect{\l\|\mathbb{E}_\alpha[\overline{x}_T] - \gamma_\alpha\r\|}\leq \sqrt{\frac{D^2(1+\ln(T))}{T}} + \frac{D}{T^{1/3}}\leq \frac{(\sqrt{2}+1)D}{T^{1/3}}
\end{equation}

To get the performance bound on $\overline{x}_T$, 
we also need the following lemma which bounds the perturbation on the gap function $\mathcal{G}(\gamma)$ defined in \eqref{eq:FW-gap}. The proof is delayed to the appendix.
\begin{lemma}\label{lem:FW-gap-bound}
For any $\gamma,\widetilde\gamma\in\mathcal{B}$,  we have
\[
\left|\mathcal{G}(\gamma) - \mathcal{G}(\widetilde\gamma)\right| \leq (2DL+M)\|\gamma-\widetilde\gamma\|.
\]
\end{lemma}

Combining Lemma \ref{lem:FW-gap-bound} with \eqref{eq:diff-bound}, we have
\[
\expect{|\mathcal{G}(\mathbb{E}_\alpha[\overline{x}_T]) - \mathcal{G}(\gamma_\alpha)| } \leq\frac{(2DL+M)(\sqrt{2}+1)D}{T^{1/3}}.
\]
Note that $\mathbb{E}_\alpha[\overline{x}_T]\in\overline{\Gamma}^*$. 
Substituting this bound and \eqref{eq:diff-bound} into Theorem \ref{col:non-convex} readily gives the following bound.

\begin{corollary}
Let $V=T^{1/3}$ and $\eta=\frac{1}{T^{2/3}}$, then, we have objective and constraint violation bounds as follows,
\begin{align*}
&\expect{\mathcal{G}(\mathbb{E}_\alpha[\overline{x}_T])} \leq \l( 2K + B^2 + \frac{LD^2}{2} + (2DL+M)(\sqrt{2}+1)D\r)\frac{1}{T^{1/3}},\\
&\dotp{\mathbf{a}_i}{\expect{\overline{x}_T}} - b_i \leq \left(\sqrt{2MD+4K + B^2 + LD^2 }+ \|\mathbf{a}_i\|(\sqrt{2}+2)D\right)\frac{1}{T^{1/3}},~\forall i\in\{1,2,\cdots,N\}.
\end{align*}
where $\mathbb E_\alpha[\cdot]$ denoted the the conditional expectation conditioned on the randomness in the first $T$ time slots generating $\gamma_\alpha$.
\end{corollary}

\subsection{Distributed non-convex stochastic optimization}
In this section, we study the problem of distributed non-convex stochastic optimization over a connected network of $K$ nodes without a central controller, and show that our proposed algorithm can be applied in such a scenario with theoretical performance guarantees.

Consider an undirected connected graph $(\mathcal  V, \mathcal  E)$, where $\mathcal V=\{1,2,\cdots,K\}$ is a set of $K$ nodes, $\mathcal  E=\{e_{ij}\}_{i,j\in\mathcal  V}$ is a collection of undirected edges, $e_{ij}=1$ if there exists an undirected edge between node $i$ and node $j$, and $e_{ij}=0$ otherwise. Two nodes $i$ and $j$ are said to be neighbors of each other if $e_{ij} = 1$. Each node holds a local vector $x^{(i)}_t\in\mathcal{X}^{(i)}_t\subseteq\mathbb{R}^{d_i}$. Let $\Gamma_i^*$ be the set of all possible ``one-shot'' expectations $\expect{x^{(i)}_t}$ achieved by any randomized stationary algorithms of choosing $x^{(i)}_t$ and let $\overline{\Gamma}_i^*$ be the closure of $\Gamma_i^*$. In addition, these nodes must collectively choose one $\theta\in\Theta\subseteq \mb R^p$, where $\Theta$ is a compact set. The goal is to solve the following problem\\
\begin{align}
\min~~&\sum_{i=1}^Kf^{(i)}\l(\gamma^{(i)},\theta\r) \label{dist-opt1}\\
s.t.~~&\gamma^{(i)}\in\overline{\Gamma}_i^*,~\forall i\in\{1,2,\cdots,K\},~\theta\in\Theta,\label{dist-opt2}
\end{align}
where $f^{(i)}:\mb R^{d_i}\times \mb R^p\rightarrow \mb R$ is a local function at node $i$, which can be non-convex on both $\gamma^{(i)}$ and $\theta$. The main difficult is that each node only knows its own $f^{(i)}$ and can only communicate with its neighbors whereas a global $\theta$ has to be chosen jointly by all nodes.
A classical way of rewriting (\ref{dist-opt1}-\ref{dist-opt2}) in a ``distributed fashion'' is to introduce a ``local copy'' $\theta^{(i)}$ of $\theta$ for each node $i$ and solve the following problem with consensus constraints.
\begin{align}
\min~~&\sum_{i=1}^Kf^{(i)}\l(\gamma^{(i)},\theta^{(i)}\r) \nonumber\\
s.t.~~&\gamma^{(i)}\in\overline{\Gamma}_i^*,~\forall i\in\{1,2,\cdots,K\},~\theta^{(i)}\in\Theta,\nonumber\\
&\theta^{(i)} = \theta^{(j)},~\text{if}~e_{ij} =1,~\forall i,j\in\{1,2,\cdots,K\}.\label{dist-opt11}
\end{align}

Note that his problem fits into the framework of (\ref{prob-1}-\ref{prob-3}) by further rewriting \eqref{dist-opt11} as a collection of inequality constraints:
\[
\theta^{(i)} \leq \theta^{(j)},~\theta^{(j)} \leq \theta^{(i)},~\text{if}~e_{ij} =1,~\forall i,j\in\{1,2,\cdots,K\}.
\]
Let $\mathcal{N}^{(i)}$ be the set of neighbors around node $i$, we can then group these constraints node-wise and write our program as follows:
\begin{align}
\min~~&\sum_{i=1}^Kf^{(i)}\l(\gamma^{(i)},\theta^{(i)}\r) \label{dist-opt21}\\
s.t.~~&\gamma^{(i)}\in\overline{\Gamma}_i^*,~\forall i\in\{1,2,\cdots,K\},~\theta^{(i)}\in\Theta,\label{dist-opt22}\\
&\theta^{(i)} \leq \theta^{(j)},~\forall j\in\mathcal{N}^{(i)},~\forall i\in\{1,2,\cdots,K\}.\label{dist-opt23}
\end{align}
To apply our algorithm, for each constraint in \eqref{dist-opt23}, 
we introduce a corresponding virtual queue vector $\mathbf Q_{ij}(t)\in\mb R^p$, which is equal to 0 at $t=0$ and updated as follows:
\begin{equation}\label{queue-update-2}
\mathbf Q_{ij}(t+1)  = \max\l\{ \mathbf Q_{ij}(t) + \theta^{(i)}(t) - \theta^{(j)}(t), 0\r\},
\end{equation}
where the maximum is taken entry-wise. Then, during each time slot $t$, we solve the following optimization problem:
\begin{multline*}
\min_{x^{(i)}\in\mathcal{X}^{(i)}_t,~\theta^{(i)}\in\Theta,~i=1,2,\cdots,K}
V\sum_{i=1}^K\l(\dotp{\nabla_\gamma f^{(i)}\l(\gamma_{t-1}^{(i)},\theta_{t-1}^{(i)}\r)}{x^{(i)}}
+\dotp{\nabla_\theta f^{(i)}\l(\gamma_{t-1}^{(i)},\theta_{t-1}^{(i)}\r)}{\theta^{(i)}}\r)\\
+\sum_{i=1}^K\sum_{j\in\mathcal{N}^{(i)}}\dotp{\mathbf{Q}_{ij}(t)}{\theta^{(i)} - \theta^{(j)}},
\end{multline*}
where $\nabla_\gamma f^{(i)}$ and $\nabla_\theta f^{(i)}$ are partial derivatives regarding $\gamma$ and $\theta$ variables respectively.
This is a separable optimization problem regarding both the agents and the decision variables $x^{(i)}$ and $\theta^{(i)}$. Overall, we have the following algorithm: Let $\gamma_{-1}^{(i)} = 0\in\mathbb{R}^{d_i}$, $\beta_{-1}^{(i)} = 0\in\mathbb{R}^{p}$. At the beginning, all the nodes can observe a common random variable $\alpha$ uniformly distributed in 
$\{0,1,2,\cdots,T-1\}$,
and at each time slot $t\in\{0,1,2,\cdots,T\}$, 

\begin{enumerate}
\item Each agent $i$ observes $\mathcal{X}^{(i)}_t$ and solve for $x^{(i)}_t$ via the following:
\[
x^{(i)}_t:= \argmin_{x^{(i)}\in\mathcal{X}^{(i)}_t}\dotp{\nabla_\gamma f^{(i)}\l(\gamma_{t-1}^{(i)},\theta_{t-1}^{(i)}\r)}{x^{(i)}}
\]
\item Each agent $i$ solves for $\theta^{(i)}_t$ observing the queue states $\l\{\mathbf{Q}_{ji}(t)\r\}_{j\in\mathcal{N}^{(i)}}$ of the neighbors:
\[
\theta^{(i)}_t := \argmin_{\theta^{(i)}\in\Theta}~V\dotp{\nabla_\theta f^{(i)}\l(\gamma_{t-1}^{(i)},\beta_{t-1}^{(i)}\r)}{\theta^{(i)}} + \sum_{j\in\mathcal{N}^{(i)}}
\dotp{\mathbf{Q}_{ij}(t) - \mathbf{Q}_{ji}(t)}{\theta^{(i)}}
\]
\item Each agent $i$ updates $\gamma_{t}^{(i)} = (1-\eta)\gamma_{t-1}^{(i)} + \eta x^{(i)}_t$, $\beta_{t}^{(i)} = (1-\eta)\beta_{t-1}^{(i)} + \eta \theta^{(i)}_t$, 
$\l\{\mathbf{Q}_{ij}(t)\r\}_{j\in\mathcal{N}^{(i)}}$ via \eqref{queue-update-2}.
\end{enumerate}
We then output $\l\{\l(\gamma_\alpha^{(i)},\beta_\alpha^{(i)}\r)\r\}_{i=1}^K$ as the solution.

We have the following performance bound on the algorithm.
\begin{corollary}
Let $V=T^{1/3}$ and $\eta=\frac{1}{T^{2/3}}$, then, we have objective and constraint violation bounds as follows,
\begin{align*}
&\expect{\sum_{i=1}^K\mathcal{G}^{(i)}(\gamma_\alpha^{(i)}, \beta_\alpha^{(i)})} \leq \mathcal{O}\l(\frac{1}{T^{1/3}}\r),\\
& \l|\expect{\beta_\alpha^{(i)}} - \expect{\beta_\alpha^{(j)}}\r|\leq \mathcal{O}\l(\frac{1}{T^{1/3}}\r),~\forall i,j\in\{1,2,\cdots,N\},~\\
&\expect{\text{dist}\l(\gamma^{(i)}_\alpha, \overline{\Gamma}_i^*\r)^2}\leq \mathcal{O}\l(\frac{1}{T^{2/3}}\r),
\end{align*}
where the notation $\mathcal{O}(\cdot)$ hides a constant independent of $T$.
Thus, to achieve an $\varepsilon$-near local optimality, the convergence time is $\mathcal{O}(1/\varepsilon^3)$.
\end{corollary}

\bibliographystyle{imsart-number}
\bibliography{non-convex}

\begin{thebibliography}{23}

\bibitem{agrawal2002optimality}
\begin{binproceedings}[author]
\bauthor{\bsnm{Agrawal},~\bfnm{Rajeev}\binits{R.}} \AND
  \bauthor{\bsnm{Subramanian},~\bfnm{Vijay}\binits{V.}}
(\byear{2002}).
\btitle{Optimality of certain channel aware scheduling policies}.
In \bbooktitle{Proceedings of the Annual Allerton Conference on Communication
  Control and Computing}
\bvolume{40}
\bpages{1533--1542}.
\end{binproceedings}
\endbibitem

\bibitem{andrews2005optimal}
\begin{binproceedings}[author]
\bauthor{\bsnm{Andrews},~\bfnm{Matthew}\binits{M.}},
  \bauthor{\bsnm{Qian},~\bfnm{Lijun}\binits{L.}} \AND
  \bauthor{\bsnm{Stolyar},~\bfnm{Alexander}\binits{A.}}
(\byear{2005}).
\btitle{Optimal utility based multi-user throughput allocation subject to
  throughput constraints}.
In \bbooktitle{INFOCOM 2005. 24th Annual Joint Conference of the IEEE Computer
  and Communications Societies. Proceedings IEEE}
\bvolume{4}
\bpages{2415--2424}.
\bpublisher{IEEE}.
\end{binproceedings}
\endbibitem

\bibitem{bubeck2015convex}
\begin{barticle}[author]
\bauthor{\bsnm{Bubeck},~\bfnm{S{\'e}bastien}\binits{S.}} \betal{et~al.}
(\byear{2015}).
\btitle{Convex optimization: Algorithms and complexity}.
\bjournal{Foundations and Trends{\textregistered} in Machine Learning}
\bvolume{8}
\bpages{231--357}.
\end{barticle}
\endbibitem

\bibitem{chiang2009nonconvex}
\begin{bincollection}[author]
\bauthor{\bsnm{Chiang},~\bfnm{Mung}\binits{M.}}
(\byear{2009}).
\btitle{Nonconvex optimization for communication networks}.
In \bbooktitle{Advances in Applied Mathematics and Global Optimization}
\bpages{137--196}.
\bpublisher{Springer}.
\end{bincollection}
\endbibitem

\bibitem{eryilmaz2007fair}
\begin{barticle}[author]
\bauthor{\bsnm{Eryilmaz},~\bfnm{Atilla}\binits{A.}} \AND
  \bauthor{\bsnm{Srikant},~\bfnm{R}\binits{R.}}
(\byear{2007}).
\btitle{Fair resource allocation in wireless networks using queue-length-based
  scheduling and congestion control}.
\bjournal{IEEE/ACM Transactions on Networking (TON)}
\bvolume{15}
\bpages{1333--1344}.
\end{barticle}
\endbibitem

\bibitem{jaggi2013revisiting}
\begin{binproceedings}[author]
\bauthor{\bsnm{Jaggi},~\bfnm{Martin}\binits{M.}}
(\byear{2013}).
\btitle{Revisiting Frank-Wolfe: Projection-Free Sparse Convex Optimization.}
In \bbooktitle{ICML (1)}
\bpages{427--435}.
\end{binproceedings}
\endbibitem

\bibitem{prop-fair-down}
\begin{barticle}[author]
\bauthor{\bsnm{Kushner},~\bfnm{H.}\binits{H.}} \AND
  \bauthor{\bsnm{Whiting},~\bfnm{P.}\binits{P.}}
(\byear{Oct. 2002}).
\btitle{Asymptotic Properties of Proportional-Fair Sharing Algorithms}.
\bjournal{Proc. 40th Annual Allerton Conf. on Communication, Control, and
  Computing, Monticello, IL}.
\end{barticle}
\endbibitem

\bibitem{lacoste2016convergence}
\begin{barticle}[author]
\bauthor{\bsnm{Lacoste-Julien},~\bfnm{Simon}\binits{S.}}
(\byear{2016}).
\btitle{Convergence rate of Frank-Wolfe for non-convex objectives}.
\bjournal{arXiv preprint arXiv:1607.00345}.
\end{barticle}
\endbibitem

\bibitem{lee2005non}
\begin{barticle}[author]
\bauthor{\bsnm{Lee},~\bfnm{J-W}\binits{J.-W.}},
  \bauthor{\bsnm{Mazumdar},~\bfnm{Ravi~R}\binits{R.~R.}} \AND
  \bauthor{\bsnm{Shroff},~\bfnm{Ness~B}\binits{N.~B.}}
(\byear{2005}).
\btitle{Non-convex optimization and rate control for multi-class services in
  the Internet}.
\bjournal{IEEE/ACM transactions on networking}
\bvolume{13}
\bpages{827--840}.
\end{barticle}
\endbibitem

\bibitem{lee2006opportunistic}
\begin{barticle}[author]
\bauthor{\bsnm{Lee},~\bfnm{J-W}\binits{J.-W.}},
  \bauthor{\bsnm{Mazumdar},~\bfnm{Ravi~R}\binits{R.~R.}} \AND
  \bauthor{\bsnm{Shroff},~\bfnm{Ness~B}\binits{N.~B.}}
(\byear{2006}).
\btitle{Opportunistic power scheduling for dynamic multi-server wireless
  systems}.
\bjournal{IEEE Transactions on Wireless Communications}
\bvolume{5}
\bpages{1506--1515}.
\end{barticle}
\endbibitem

\bibitem{li2016douglas}
\begin{barticle}[author]
\bauthor{\bsnm{Li},~\bfnm{Guoyin}\binits{G.}} \AND
  \bauthor{\bsnm{Pong},~\bfnm{Ting~Kei}\binits{T.~K.}}
(\byear{2016}).
\btitle{Douglas--Rachford splitting for nonconvex optimization with application
  to nonconvex feasibility problems}.
\bjournal{Mathematical programming}
\bvolume{159}
\bpages{371--401}.
\end{barticle}
\endbibitem

\bibitem{neely2017frank}
\begin{binproceedings}[author]
\bauthor{\bsnm{Neely},~\bfnm{Michael~J.}\binits{M.~J.}}
\btitle{Optimal Convergence and Adaptation for Utility Optimal Opportunistic
  Scheduling}.
In \bbooktitle{Communication, Control, and Computing (Allerton), 2017 55th
  Annual Allerton Conference on}.
\bpublisher{IEEE}.
\end{binproceedings}
\endbibitem

\bibitem{neely2010stochastic}
\begin{barticle}[author]
\bauthor{\bsnm{Neely},~\bfnm{Michael~J}\binits{M.~J.}}
(\byear{2010}).
\btitle{Stochastic network optimization with application to communication and
  queueing systems}.
\bjournal{Synthesis Lectures on Communication Networks}
\bvolume{3}
\bpages{1--211}.
\end{barticle}
\endbibitem

\bibitem{neely2010stochastic-nonconvex}
\begin{binproceedings}[author]
\bauthor{\bsnm{Neely},~\bfnm{Michael~J}\binits{M.~J.}}
(\byear{2010}).
\btitle{Stochastic network optimization with non-convex utilities and costs}.
In \bbooktitle{Information Theory and Applications Workshop (ITA), 2010}
\bpages{1--10}.
\bpublisher{IEEE}.
\end{binproceedings}
\endbibitem

\bibitem{neely2008fairness}
\begin{barticle}[author]
\bauthor{\bsnm{Neely},~\bfnm{Michael~J}\binits{M.~J.}},
  \bauthor{\bsnm{Modiano},~\bfnm{Eytan}\binits{E.}} \AND
  \bauthor{\bsnm{Li},~\bfnm{Chih-Ping}\binits{C.-P.}}
(\byear{2008}).
\btitle{Fairness and optimal stochastic control for heterogeneous networks}.
\bjournal{IEEE/ACM Transactions On Networking}
\bvolume{16}
\bpages{396--409}.
\end{barticle}
\endbibitem

\bibitem{nesterov2015complexity}
\begin{barticle}[author]
\bauthor{\bsnm{Nesterov},~\bfnm{Yu}\binits{Y.}}
(\byear{2015}).
\btitle{Complexity bounds for primal-dual methods minimizing the model of
  objective function}.
\bjournal{Mathematical Programming}
\bpages{1--20}.
\end{barticle}
\endbibitem

\bibitem{reddi2016stochastic}
\begin{binproceedings}[author]
\bauthor{\bsnm{Reddi},~\bfnm{Sashank~J}\binits{S.~J.}},
  \bauthor{\bsnm{Sra},~\bfnm{Suvrit}\binits{S.}},
  \bauthor{\bsnm{P{\'o}czos},~\bfnm{Barnab{\'a}s}\binits{B.}} \AND
  \bauthor{\bsnm{Smola},~\bfnm{Alex}\binits{A.}}
(\byear{2016}).
\btitle{Stochastic frank-wolfe methods for nonconvex optimization}.
In \bbooktitle{Communication, Control, and Computing (Allerton), 2016 54th
  Annual Allerton Conference on}
\bpages{1244--1251}.
\bpublisher{IEEE}.
\end{binproceedings}
\endbibitem

\bibitem{stolyar2005asymptotic}
\begin{barticle}[author]
\bauthor{\bsnm{Stolyar},~\bfnm{Alexander~L}\binits{A.~L.}}
(\byear{2005}).
\btitle{On the asymptotic optimality of the gradient scheduling algorithm for
  multiuser throughput allocation}.
\bjournal{Operations research}
\bvolume{53}
\bpages{12--25}.
\end{barticle}
\endbibitem

\bibitem{stolyar2005maximizing}
\begin{barticle}[author]
\bauthor{\bsnm{Stolyar},~\bfnm{Alexander~L}\binits{A.~L.}}
(\byear{2005}).
\btitle{Maximizing queueing network utility subject to stability: Greedy
  primal-dual algorithm}.
\bjournal{Queueing Systems}
\bvolume{50}
\bpages{401--457}.
\end{barticle}
\endbibitem

\bibitem{wang2015global}
\begin{barticle}[author]
\bauthor{\bsnm{Wang},~\bfnm{Yu}\binits{Y.}},
  \bauthor{\bsnm{Yin},~\bfnm{Wotao}\binits{W.}} \AND
  \bauthor{\bsnm{Zeng},~\bfnm{Jinshan}\binits{J.}}
(\byear{2015}).
\btitle{Global convergence of ADMM in nonconvex nonsmooth optimization}.
\bjournal{arXiv preprint arXiv:1511.06324}.
\end{barticle}
\endbibitem

\bibitem{wei2015probabilistic}
\begin{barticle}[author]
\bauthor{\bsnm{Wei},~\bfnm{Xiaohan}\binits{X.}},
  \bauthor{\bsnm{Yu},~\bfnm{Hao}\binits{H.}} \AND
  \bauthor{\bsnm{Neely},~\bfnm{Michael~J}\binits{M.~J.}}
(\byear{2015}).
\btitle{A Probabilistic Sample Path Convergence Time Analysis of
  Drift-Plus-Penalty Algorithm for Stochastic Optimization}.
\bjournal{arXiv preprint arXiv:1510.02973}.
\end{barticle}
\endbibitem

\bibitem{yang2017alternating}
\begin{barticle}[author]
\bauthor{\bsnm{Yang},~\bfnm{Lei}\binits{L.}},
  \bauthor{\bsnm{Pong},~\bfnm{Ting~Kei}\binits{T.~K.}} \AND
  \bauthor{\bsnm{Chen},~\bfnm{Xiaojun}\binits{X.}}
(\byear{2017}).
\btitle{Alternating direction method of multipliers for a class of nonconvex
  and nonsmooth problems with applications to background/foreground
  extraction}.
\bjournal{SIAM Journal on Imaging Sciences}
\bvolume{10}
\bpages{74--110}.
\end{barticle}
\endbibitem

\bibitem{yu2017online}
\begin{binproceedings}[author]
\bauthor{\bsnm{Yu},~\bfnm{Hao}\binits{H.}},
  \bauthor{\bsnm{Neely},~\bfnm{Michael}\binits{M.}} \AND
  \bauthor{\bsnm{Wei},~\bfnm{Xiaohan}\binits{X.}}
(\byear{2017}).
\btitle{Online Convex Optimization with Stochastic Constraints}.
In \bbooktitle{Advances in Neural Information Processing Systems}
\bpages{1427--1437}.
\end{binproceedings}
\endbibitem

\end{thebibliography}

\section{Appendix}
\begin{proof}[Proof of Lemma \ref{lem:drift-2}]
First of all, rearranging \eqref{eq:unconditional} with the fact that $\Delta(t) =\frac12( \|\mf Q(t+1)\|^2 - \|\mf Q(t)\|^2)$ gives
\begin{multline*}
\|\mf Q(t+1)\|^2 - \|\mf Q(t)\|^2 \leq
\frac{2V}{\eta}\l( f(\gamma_{t-1}) - f(\gamma_{t})  \r) +
2V\dotp{\nabla f(\gamma_{t-1})}{x_t-\gamma_{t-1}} \\
+ 2\dotp{\mathbf Q(t)}{\mathbf{A}x_t-\mathbf{b}}
+VL\eta D^2 + B^2.
\end{multline*}
Taking the telescoping sums from $t$ to $t+t_0-1$ and taking conditional expectation from both sides conditioned on $\mathcal{H}_t$, where $\mathcal{H}_t$ is the system history up to time slot $t$ including $\mathbf{Q}(t)$, give
\begin{multline}\label{eq:sum-bound}
\expect{\|\mf Q(t+t_0)\|^2 - \|\mf Q(t)\|^2 |~\mathcal{H}_t} \leq
VL\eta D^2t_0 + B^2t_0+
\frac{2V}{\eta}\expect{ f(\gamma_{t-1}) - f(\gamma_{t+t_0-1}) | ~\mathcal{H}_t  } \\
+
2\expect{\left.\sum_{\tau = t}^{t+t_0-1}\l(V\dotp{\nabla f(\gamma_{\tau-1})}{x_{\tau}-\gamma_{\tau-1}} 
+ \dotp{\mathbf Q(\tau)}{\mathbf{A}x_{\tau}-\mathbf{b}}\r)\right|~\mathcal{H}_t}.\\
\leq VL\eta D^2t_0 + B^2t_0+\frac{4VK}{\eta}
+2\expect{\left.\sum_{\tau = t}^{t+t_0-1}\l(V\dotp{\nabla f(\gamma_{\tau-1})}{x_{\tau}-\gamma_{\tau-1}} 
+ \dotp{\mathbf Q(\tau)}{\mathbf{A}x_{\tau}-\mathbf{b}}\r)\right|~\mathcal{H}_t}
\end{multline}
To bound the last term on the right hand side, we use the tower property of the conditional expectation that for any $\tau\geq t$, 
\begin{align}
&\expect{\left.V\dotp{\nabla f(\gamma_{\tau-1})}{x_{\tau}-\gamma_{\tau-1}} 
+ \dotp{\mathbf Q(\tau)}{\mathbf{A}x_{\tau}-\mathbf{b}}\right|~\mathcal{H}_t}\nonumber\\
=&\expect{\left.\expect{V\dotp{\nabla f(\gamma_{\tau-1})}{x_{\tau}-\gamma_{\tau-1}} 
+ \dotp{\mathbf Q(\tau)}{\mathbf{A}x_{\tau}-\mathbf{b}} |~\mathcal{H}_{\tau}} \right|~\mathcal{H}_t}. \label{inter-t}
\end{align}
Since the proposed algorithm chooses $\mf x_{\tau}$ to minimize $V\dotp{\nabla f(\gamma_{\tau-1})}{x_{\tau}-\gamma_{\tau-1}} 
+ \dotp{\mathbf Q(\tau)}{\mathbf{A}x_{\tau}-\mathbf{b}}$ given $\mf Q(\tau)$, it must achieve less value than that of any randomized stationary policy. Specifically, it dominates $\widetilde{x}_\tau$ satisfying the Slater's condition (Assumption \ref{ass:slater}). This implies,
\begin{align*}
&\expect{V\dotp{\nabla f(\gamma_{\tau-1})}{x_{\tau}-\gamma_{\tau-1}} 
+ \dotp{\mathbf Q(\tau)}{\mathbf{A}x_{\tau}-\mathbf{b}} |~\mathcal{H}_{\tau}}\\
\leq&\expect{V\dotp{\nabla f(\gamma_{\tau-1})}{\widetilde{x}_{\tau}-\gamma_{\tau-1}} 
+ \dotp{\mathbf Q(\tau)}{\mathbf{A}\widetilde{x}_{\tau}-\mathbf{b}} |~\mathcal{H}_{\tau}}\\
=&\expect{V\dotp{\nabla f(\gamma_{\tau-1})}{\widetilde{x}_{\tau}-\gamma_{\tau-1}}  |~\mathcal{H}_{\tau}} 
+ \dotp{\mathbf Q(\tau)}{\expect{\mathbf{A}\widetilde{x}_{\tau}-\mathbf{b} |~\mathcal{H}_{\tau}}} \\
\leq&\expect{V\dotp{\nabla f(\gamma_{\tau-1})}{\widetilde{x}_{\tau}-\gamma_{\tau-1}}  |~\mathcal{H}_{\tau}}  - \varepsilon\sum_{i=1}^NQ_i(\tau)\leq VMD - \varepsilon\|\mf Q(\tau)\|,
\end{align*}
where the second inequality follows from Assumption \ref{ass:slater}, that 
$\expect{\mathbf{A}\widetilde{x}_{\tau}-\mathbf{b} |~\mathcal{H}_{\tau}} = \expect{\mathbf{A}\widetilde{x}_{\tau}-\mathbf{b}}\leq -\varepsilon \textbf{1}$ because the randomized stationary policy $\widetilde{x}_{\tau}$ is independent of $\mathcal{H}_{\tau}$, and the third inequality follows from 
$\|\nabla f(\gamma_{\tau-1})\|\leq M$ and $\|\widetilde{x}_{\tau}-\gamma_{\tau-1}\|\leq D$. By Triangle inequality, we have$ \|\mf Q(\tau)\|\geq \|\mf Q(t)\| - \|\mf Q(\tau)-\mf Q(t)\|$ and
\begin{multline*}
\expect{V\dotp{\nabla f(\gamma_{\tau-1})}{x_{\tau}-\gamma_{\tau-1}} 
+ \dotp{\mathbf Q(\tau)}{\mathbf{A}x_{\tau}-\mathbf{b}} |~\mathcal{H}_{\tau}}
\leq VMD - \varepsilon\|\mf Q(t)\| + \varepsilon\|\mf Q(\tau) - \mf Q(t)\|\\
\leq VMD - \varepsilon\|\mf Q(t)\|  + \varepsilon B(\tau-t),
\end{multline*}
where we use the bound $\|\mf Q(\tau+1) - \mf Q(\tau)\|\leq \|\mf A x_{\tau} - \mf b\|\leq B$ for any $\tau$.
Substituting this bound into \eqref{inter-t} gives
\[
\expect{\left.V\dotp{\nabla f(\gamma_{\tau-1})}{x_{\tau}-\gamma_{\tau-1}} 
+ \dotp{\mathbf Q(\tau)}{\mathbf{A}x_{\tau}-\mathbf{b}}\right|~\mathcal{H}_t}
\leq VMD - \varepsilon\|\mf Q(t)\|  + \varepsilon B(\tau-t).
\]
Substituting this bound into \eqref{eq:sum-bound}, we get 
\begin{align*}
&\expect{\|\mf Q(t+t_0)\|^2 - \|\mf Q(t)\|^2 |~\mathcal{H}_t}\\
 \leq&
VL\eta D^2t_0 + B^2t_0+ 2VMDt_0+\frac{4VK}{\eta} +2\varepsilon B\sum_{\tau = t}^{t+t_0-1}(\tau-t)- 2\varepsilon t_0\|\mf Q(t)\|\\
\leq&VL\eta D^2t_0 + B^2t_0+2VMDt_0+\frac{4VK}{\eta} + \varepsilon Bt_0^2- 2\varepsilon t_0 \|\mf Q(t)\|.
\end{align*}
Suppose $\|\mathbf{Q}(t)\|> \lambda = \frac{(VL\eta D^2 + B^2+2VMD)t_0+4VK/\eta + \varepsilon Bt_0^2+\varepsilon^2t_0^2}{\varepsilon t_0}$, then, 
\[
VL\eta D^2t_0 + B^2t_0+2VMDt_0+\frac{4VK}{\eta} + \varepsilon Bt_0^2\leq \varepsilon t_0\|\mf Q(t)\|,
\]
and it follows
\[
\expect{\|\mf Q(t+t_0)\|^2 - \|\mf Q(t)\|^2 |~\mathcal{H}_t}\leq -\varepsilon t_0\|\mf Q(t)\|.
\]
Since $\expect{\|\mf Q(t+t_0)\|^2 - \|\mf Q(t)\|^2 |~\mathcal{H}_t} = \expect{\|\mf Q(t+t_0)\|^2  |~\mathcal{H}_t}- \|\mf Q(t)\|^2$, rearranging the terms using the fact that $\|\mf Q(t)\|\geq \varepsilon t_0$,
\[
\expect{\|\mf Q(t+t_0)\|^2  |~\mathcal{H}_t} \leq \l(  \|\mf Q(t)\| - \varepsilon t_0/2 \r)^2.
\]
Taking square root from both sides and using Jensen's inequality,
\[
\expect{\|\mf Q(t+t_0)\| |~\mathcal{H}_t}\leq  \|\mf Q(t)\| -  \varepsilon t_0/2.
\]
On the other hand, we always have 
\[
\l| \|\mf Q(t+1)\| - \|\mf Q(t)\|  \r|\leq \|\mf A x_t-\mf b\|\leq B,~\forall t, 
\]
finishing the proof.
\end{proof}

\begin{proof}[Proof of Lemma \ref{lem:FW-gap-bound}]
First, we have
\begin{align*}
\mathcal{G}(\gamma)
&\leq \mathcal{G}(\widetilde\gamma)+
 \sup_{v\in\overline{\Gamma}^*,~\mathbf{A}v\leq\mathbf{b}}\left|\dotp{\nabla f(\gamma)}{\gamma - v} - \dotp{\nabla f(\widetilde\gamma)}{\widetilde\gamma - v}\right|\\
&\leq {\mathcal{G}(\widetilde\gamma)}+ { \sup_{v\in\overline{\Gamma}^*,~\mathbf{A}v\leq\mathbf{b}}\left|\dotp{\nabla f(\gamma)- \nabla f(\widetilde\gamma)}{v}\right|}
+{|\dotp{\nabla f(\gamma)}{\gamma} - \dotp{\nabla f(\widetilde\gamma)}{\widetilde\gamma}|}\\
&\leq {\mathcal{G}(\widetilde\gamma)}
+  \underbrace{{ \sup_{v\in\overline{\Gamma}^*,~\mathbf{A}v\leq\mathbf{b}}\left|\dotp{\nabla f(\gamma)- \nabla f(\widetilde\gamma)}{v}\right|}}_{(I)}
+\underbrace{{|\dotp{\nabla f(\gamma)}{\gamma- \widetilde\gamma}|}}_{(II)} 
+ \underbrace{{| \dotp{\nabla f(\gamma) - \nabla f(\widetilde\gamma)}{\widetilde\gamma} |}}_{(III)}.
\end{align*}
We have by Lemma \ref{lem:path-average},
\[
(I)\leq D\cdot {\|\nabla f(\gamma)- \nabla f(\widetilde\gamma)\|}\leq DL\cdot {\|\gamma - \widetilde\gamma\|}.
\]
Similarly, we have by Cauchy-Schwarz inequality, 
\begin{align*}
(II) \leq& {\|\nabla f(\gamma)\|\cdot \|\gamma - \widetilde\gamma  \|}\leq M\|\gamma - \widetilde\gamma  \|,\\
(III)\leq&{\|\nabla f(\gamma) - \nabla f(\widetilde\gamma)\|\cdot \|\widetilde\gamma  \|},
\leq DL\cdot{\|\gamma - \widetilde\gamma  \|}.
\end{align*}
Overall, we get
\[
\mathcal{G}(\gamma)  \leq  {\mathcal{G}(\widetilde\gamma)}+ (2DL+M)\|\gamma-\widetilde\gamma\|.
\]
Exchanging the position of $\gamma$ and $\widetilde{\gamma}$, and repeating the argument give the result.
\end{proof}

\end{document}